\documentclass[12pt,reqno]{amsart}
\usepackage[T1]{fontenc}
\usepackage{amsfonts}
\usepackage{amssymb}
\usepackage{mathrsfs}
\usepackage[latin1]{inputenc}
\usepackage{amsmath}
\usepackage[english]{babel}
\usepackage{setspace}
\usepackage{latexsym,amsfonts,amsmath,amssymb,soul}

 \newtheorem{thm}{Theorem}[section]
 \newtheorem{cor}[thm]{Corollary}
 \newtheorem{lem}[thm]{Lemma}
 \newtheorem{prop}[thm]{Proposition}
 \theoremstyle{definition}
 
 \newtheorem{rem}[thm]{Remark}
 \newtheorem{ex}[thm]{Examples}
 \numberwithin{equation}{section}

 \numberwithin{equation}{section}

\newcommand{\R}{{\mathbb R}}
\newcommand{\D}{{\mathbb D}}

\newcommand{\C}{{\mathbb C}}
\newcommand{\N}{{\mathbb N}}
\newcommand{\T}{{\mathbb T}}
\newcommand{\cD}{{\mathcal D}}
\newcommand{\cL}{{\mathcal L}}
\newcommand{\cF}{{\mathcal F}}
\newcommand{\cE}{{\mathcal E}}

\newcommand{\cC}{{\mathcal C}}

\newcommand{\cM}{{\mathcal M}}

\newcommand{\cS}{{\mathcal S}}
\newcommand{\cO}{{\mathcal O}}

\newcommand{\su}{\subseteq}

\newcommand\proj{\mathop{\rm proj\,}}
\newcommand\ind{\mathop{\rm ind\,}}
\newcommand\Ker{\mathop{\rm Ker}}

\begin{document}

%
%
%
%

\title[Spectra and ergodic properties]
 { Spectra and ergodic properties of multiplication and convolution operators on the space  $\cS(\R)$}

\author[A.A. Albanese, C. Mele]{Angela\,A. Albanese and Claudio Mele}

\address{ Angela A. Albanese\\
Dipartimento di Matematica e Fisica ``E. De Giorgi''\\
Universit\`a del Salento- C.P.193\\
I-73100 Lecce, Italy}
\email{angela.albanese@unisalento.it}

\address{Claudio Mele\\
	Dipartimento di Matematica e Fisica ``E. De Giorgi''\\
	Universit\`a del Salento- C.P.193\\
	I-73100 Lecce, Italy}
\email{claudio.mele1@unisalento.it}





\begin{abstract} In this paper we investigate the spectra and the ergodic  properties of the multiplication operators and the convolution  operators
	acting on the Schwartz space $\cS(\R)$ of rapidly decreasing functions, i.e., operators of
	the form $M_h: \cS(\R)\to\cS(\R)$, $f  \mapsto  h f $, and $C_T\colon \cS(\R)\to\cS(\R)$, $f\mapsto T\star f$. Precisely, we determine their spectra and   characterize when those operators are power bounded and mean ergodic. 
\end{abstract}

\maketitle
\section{Introduction }\label{intro}
Convolution  operators, as well as multiplication operators, have been  intensively studied 
  in spaces of  functions or distributions, from different point of views. For instance, the problem 
  to characterize when the multiplication operator acting on smooth functions has closed range has attracted the attention of several authors (see \cite{BoFrJo} and the references therein) and is still open.  The problem is 
  equivalent to the well-known  division problem for distributions posed by L. Schwartz \cite[Chap. 5, Section 5]{Sh}.
  However, it seems that power boundedness
  and (uniform) mean ergodicity of multiplication and convolution operators on the  Schwartz space $\cS(\R^N)$
  of rapidly decreasing functions  had not been
  investigated. 
  
  The Schwartz space $\cS(\R^N)$
  of rapidly decreasing functions is the most important space of classical analysis besides the space of smooth functions and the space of real analytic functions.
  The   multipliers of $\cS(\R^N)$ are the functions
  $h\in C^\infty(\R^N)$ such that the multiplication operator $M_h\colon \cS(\R^N)\mapsto \cS(\R^N)$, $f\mapsto hf$,
  is well defined and continuous. The space of all multipliers is denoted by $\cO_M(\R^N)$.  The convolutors of  $\cS(\R^N)$ are the distributions $T\in \cS'(\R^N)$ such that the convolution operator $C_T\colon \cS(\R^N)\to \cS(\R^N)$, $f\mapsto T\star f$, is well defined and continuous. The space of all convolutors is denoted by $\cO'_C(\R^N)$. In \cite{BoFrJo} the authors  characterized
  the multipliers $h\in \cO_M(\R)$ such that $M_h\colon \cS(\R)\to \cS(\R)$
  has closed range. We also mention that in the last years  the study of the properties, like closed range and dynamical behaviour, of the 
    composition operators acting 
  on  the  Schwartz space $\cS(\R)$  has been considered by 
  several authors (see  \cite{FeGaJo,FeGaJo-2,GaJo,GoPr} for examples and the references therein).
  
  In this paper we study the spectra and the ergodic properties of the multiplication and the convolution operators defined in the Schwartz space
  $\cS(\R)$ of one variable rapidly decreasing functions. We determine their spectra and characterize when those operators are power bounded and (uniformly) mean ergodic. In
  particular, we show that the spectra of the  multiplication operator $M_h$ (the convolution operator $C_T$, resp.) acting either on $\cS(\R)$ or on $\cO_M(\R)$ (on $\cO_C(\R)$, the strong dual of $\cO_C'(\R)$, resp.) coincide, Theorem  \ref{T.Spectra} and \ref{T.SpectraC}. We prove that
   the multiplication operator $M_h$ is power bounded (uniformly mean ergodic, resp.) when it acts on $\cS(\R)$ if and only if it is power bounded (uniformly mean ergodic, resp.) when it acts on $\cO_M(\R)$, Theorems \ref{T.PowerB} and \ref{T.MeanERg}.  These conditions are expressed in terms  of the multiplier $h$.  Similar characterizations are also given in the case of convolution operators. The properties of the Fourier transfom allows to reduce the proofs to the multiplication operator case. Precisely, we show that the convolution operator $C_T$ is power bounded (uniformly mean ergodic, resp.) when it acts on $\cS(\R)$ if and only if  it is power bounded (uniformly mean ergodic, resp.) when it acts on $\cO_C(\R)$, Propositions \ref{P.PBC} and \ref{P.PBM}. These conditions are expressed in terms  of the convolutor $T$.
   
   We also present characterizations of the power boundedness and the (uniform) mean ergodicity of the multiplication operators and the convolution operators when those  act on the space $C^\infty(\R)$, Propositions \ref{P.CInftyM} and \ref{P.CInftyC}. 
   
   The paper ends with an Appendix, where we collect some general  results on the spectrum of operators acting on Fr\'echet spaces.
  
  We remark that the description of the spectra and  the characterizations of the power boundedness and the mean ergodicity   are valid for
  the several variables case with the same proofs. We have considered only the one dimensional case for the sake of simplicity in computations of derivatives to describe the spectra.

\section{Preliminaries}

In this section, we first recall some general notation and results on operators in locally convex spaces.

Let $E$ be a locally convex Hausdorff space (briefly, lcHs) and let $\cL(E)$ denote the space of all continuous linear operators from $E$ into itself. Given $T\in \cL(E)$, the \textit{resolvent set} of $T$ is defined by 
\[
\rho(T):=\{\lambda \in \C\colon \lambda I-T\colon E\to E \ {\rm is\ bijective\ and \ } (\lambda I-T)^{-1}\in \cL(E)\}
\]
and the \textit{spectrum} of $T$ is defined by $\sigma(T):=\C\setminus \rho(T)$. For $\lambda\in \rho(T)$ we define $R(\lambda,T):=(\lambda I-T)^{-1}$ which is called the \textit{resolvent operator of $T$ at $\lambda$}. The \textit{point spectrum} is defined by 
\[
\sigma_p(T):=\{\lambda\in\C\colon \lambda I-T {\rm \ is\ not \ injective}\}. 
\]
Whenever $\lambda,\mu\in\rho(T)$ we have the \textit{resolvent identity} $R(\lambda,T)-R(\mu,T)=(\mu-\lambda)R(\lambda,T)R(\mu,T)$. Unlike for Banach spaces, it may happens that $\rho(T)=\emptyset$ or that $\rho(T)$ is not open in $\C$ (see, f.i., \cite{ABR-3}). This is the reason for which many authors consider the subset $\rho^*(T)$ of $\rho(T)$ consisting of all $\lambda\in\C$ for which there exists $\delta>0$ such that $B(\lambda,\delta):=\{\mu\in\C\colon |\mu-\lambda|<\delta\}\subseteq \rho(T)$ and the set $\{R(\mu,T)\colon \mu\in B(\lambda,\delta)\}$ is equicontinuous in $\cL(E)$. If $E$ is a Fr\'echet space, then it
suffices that this set is bounded in $\cL_s (E)$, where $\cL_s(E)$  denotes $\cL(E)$ endowed with the
strong operator topology, i.e., the topology of uniform convergence on the finite subsets of $E$. The advantange of $\rho^*(T)$, whenever it is not empty, is that it is open and the resolvent map $R\colon \lambda\mapsto R(\lambda,T)$ is holomorphic from $\rho^*(T)$ into $\cL_b(E)$ (see, f.i., \cite[Propositione 3.4]{ABR}), where $\cL_b(E)$ denotes $\cL(E)$ endowed with the topology of uniform convergence on the
bounded subsets of $E$. Define $\sigma^*(T) := \C\setminus\rho^*(T)$,
which is a closed set containing $\sigma(T)$. In \cite[Remark 3.5(vi)]{ABR} an example of an operator $T\in \cL(E)$,
with $E$ a Fr\'echet space, is presented such that $\overline{\sigma(T)}\subsetneq \sigma^*(T)$.

For further basic properties of the resolvent set and the resolvent map we refer to
\cite{Va,W} for operators on locally convex spaces.

An operator $T\in \cL(E)$, with $E$ a lcHs, is called \textit{power bounded} if $\{T^n\}_{n\in\N}$ is an
equicontinuous subset of $\cL(E)$.

The  Ces\`aro means of an operators $T\in\cL(E)$, with $E$ a lcHs, are defined by
\[
T_{[n]} :=\frac{1}{n}\sum_{m=1}^nT^m,\quad n\in\N.
\]
The operator  $T$ is called \textit{mean ergodic} (resp. \textit{uniformly mean ergodic}) if $\{T_{[n]}\}_{n\in\N}$ 
is a convergent
sequence in $\cL_s(E)$ (resp. in $\cL_b(E)$). 
The  Ces\`aro means of $T$ satisfies the following identities
\[ 
\frac{T^n}{n}= T_{[n]}-\frac{n-1}{n} T_{[n-1]}, \quad  n \geq 2.
\] 
So, it is clear that $\frac{T^n}{n}\to 0$ in $\cL_s(E)$ as $n\to\infty$, whenever $T$ is mean ergodic.
Furthemore, if $E$ is a barrelled lcHs space and $T\in\cL(E)$ is mean ergodic, the operator $P:=\lim_{n\to\infty}T_{[n]}$ in $\cL_s(E)$ is a projection on $E$ satisfying ${\rm Im}\, P = \ker(I- T)$ and $\ker P =\overline{{\rm Im}\,(I-T)}$ with
\[
X = \overline{{\rm Im}\,(I-T)}\oplus \ker(I-T).
\] 
If  $E$ is a Montel lcHs, i.e., a barrelled lcHs such that every bounded set is relatively compact, then the operator $T$ is uniformly mean ergodic whenever it is mean ergodic. Furthemore, in reflexive Fr\'echet spaces (in Montel Fr\'echet spaces, resp.) every power
bounded operator is necessarily mean ergodic (uniformly mean ergodic, resp.) \cite[Corollary 2.7, Proposition 2.9]{ABR-0}. The converse is not true in general, see, f.i., \cite[\S 6]{H}.

For further results on mean ergodic operators we refer to \cite{Kr,Y}. For recent results on mean ergodic operators in
lcHs' we refer to \cite{ABR-0,ABR-J,ABR-JJ,ABR,ABR-2,P}, for example, and the references therein.

We now recall the necessary definitions and some basic properties of the space $\cS(\R)$ and the spaces $\cO_M(\R)$ and $\cO_C(\R)$.

The space $\cS(\R)$ of rapidly decreasing functions on $\R$ is defined by 
 \[
 \cS(\R)=\{f\in C^\infty(\R)\colon ||f||_n:=\sup_{x\in\R}\sup_{i=0,\ldots, n}(1+x^2)^n|f^{(i)}(x)|<+\infty {\rm \ for \ every\ } n\in\N_0\}.
 \]
The space $\cS(\R)$ is a nuclear Fr\'echet space and hence, it is Montel  and reflexive. Accordingly, its strong dual $\cS'(\R)$ is  a nuclear lcHs. In particular, $\cS'(\R)$ is barrelled and bornological lcHs. 

The space  of multipliers $\cO_M(\R)$  of slowly increasing functions on $\R$ is given by
\[
\cO_M(\R)=\cap_{m=1}^\infty\cup_{n=1}^\infty\{f\in C^\infty(\R)\colon |f|_{m,n}:=\sup_{x\in\R}\sup_{0\leq i\leq m}(1+x^2)^{-n}|f^{(i)}(x)|<\infty\},
\]
where $\cO_{n}^m(\R):=\{f\in C^\infty(\R)\colon |f|_{m,n}:=\sup_{x\in\R}\sup_{0\leq i\leq m}(1+x^2)^{-n}|f^{(i)}(x)|<\infty\}$, endowed with the norm $|\cdot |_{m,n}$, is a Banach space for any $m,n\in\N$. The elements of   $\cO_M(\R)$ are called   slowly increasing functions on $\R$. 
The space $\cO_M(\R)$, endowed with its natural lc-topology, i.e., $\cO_M(\R) =\proj_{\stackrel{m}{\leftarrow}}\ind_{\stackrel{n}{\to}}\cO_{n}^m(\R)$, is a projective limit of complete (LB)-spaces. In particular, $\cO_M(\R)$ is a bornological nuclear lcHs (hence, Montel and reflexive), see \cite{Gro}. Its strong dual $\cO'_M(\R)$ is also a bornological nuclear lcHs.
Furthemore, a fundamental system of continuous norms on $\cO_M(\R)$ is given by
\[
p_{m,v}(f)=\sup_{x\in\R}\sup_{0\leq i\leq m}|v(x)||f^{(i)}(x)|,\quad f\in \cO_M(\R),
\]
where $v\in \cS(\R)$ and $m\in\N$ (see, f.i., \cite{Ch}). Since $(\cO_M(\R), \cdot)$ is an algebra, we have then  for every $m\in\N$ and $v\in \cS(\R)$ that there exist $v_1,v_2\in\cS(\R)$ and $m'\in\N$ with $m'\geq m$ such that
\begin{equation}\label{eq.Prod}
p_{m,v}(fg)\leq p_{m',v_1}(f)p_{m',v_2}(g),\quad f,g\in\cO_M(\R).
\end{equation} 
The space  $\cO_C(\R)$ of  very slowly increasing functions on $\R$ is given by
\[
\cO_C(\R)=\cup_{n=1}^\infty\cap_{m=1}^\infty\{f\in C^\infty(\R)\colon |f|_{m,n}:=\sup_{x\in\R}\sup_{0\leq i\leq m}(1+x^2)^{-n}|f^{(i)}(x)|<\infty\}.
\]
The space $\cO_C(\R)$, endowed with its natural lc-topology, i.e., $\cO_C(\R) =\ind_{\stackrel{n}{\to}}\proj_{\stackrel{m}{\leftarrow}}\cO_{n}^m(\R)$, is a  complete (LF)-space. In particular, $\cO_C(\R)$ is a bornological nuclear lcHs (hence, Montel and reflexive), see \cite{Gro}. Its strong dual $\cO'_C(\R)$ is the space of all convolutors of $\cS(\R)$. In particular,  $\cO'_C(\R)$  is also a bornological nuclear lcHs.

The space $\cO_M(\R)$ is the space of multipliers of $\cS(\R)$ and its strong dual $\cS'(\R)$. So, for any fixed $h\in\cO_M(\R)$, the multiplication operator $M_h\colon \cS(\R)\to \cS(\R)$, $f\mapsto hf$, is continuous and hence, its transpose $\cM_h:=M_h'\colon \cS'(\R)\to \cS'(\R)$, $S\mapsto hS$, is also continuous. Furthemore, for any $h\in\cO_M(\R)$, the multiplication operator $M_h$ ($\cM_h$, resp.) acts continuosly from $\cO_M(\R)$ into itself (from $\cO'_M(\R)$ into itself, resp.).

The space $\cO_C'(\R)$  is the space of convolutors of $\cS(\R)$ and its strong dual $\cS'(\R)$. Precisely, 
if $T\in \cO'_C(\R)$ and $f\in\cS(\R)$, then the convolution $T\star f$ defined by 
\[
(T\star f)(x):=\langle T_y,\check{\tau_x f}\rangle, \quad x\in\R,
\]
is a function of $\cS(\R)$, where  $(\tau_xf)(y):=f(x+y)$ for $x,y\in\R$ and $\check{f}(y):=f(-y)$ for $y\in\R$. While, if $T\in \cO'_C(\R)$ and $S\in \cS'(\R)$, then the convolution $T\star S$ defined by 
\[
(T\star S)(f)=\langle S, \check{T}\star f\rangle,\quad f\in\cS(\R),
\]
belongs to $\cS'(\R)$, where $\check{T}$ denotes  the distribution defined by $\varphi\mapsto \langle \check{T}, \varphi\rangle:=\langle T, \check{\varphi}\rangle$. We point out that in case $S$ belongs to $\in \cO_C'(\R)$, the convolution  $T\star S\in \cO'_C(\R)$ too. 

For any fixed $T\in \cO'_C(\R)$,  the convolution  operator $C_T\colon \cS(\R)\to \cS(\R)$ (the convolution operator $C_T\colon \cO_C(\R)\to \cO_C(\R)$, resp.), $f\mapsto T\star f$, is a continuous linear operator from $\cS(\R)$  into itself (from $\cO_C(\R)$ into itself, resp.) and hence, its transpose $\cC_T:=C_T'\colon \cS'(\R)\to \cS'(\R)$ ($\cC_T:=C_T'\colon \cO_C'(\R)\to \cO_C'(\R)$, resp.), $S\mapsto T\star S$, is also a continuous linear operator from $\cS'(\R)$ into itself (from $\cO'_C(\R)$ into itself, resp.). 

Through the paper, we consider the  following notation for the Fourier transform of a function $f\in L^1(\R)$:
\[
\hat{f}(\xi):=\cF(f)(\xi)=\int_{\R}e^{-ix\xi}f(x)\,dx, \quad \xi\in\R.
\]
The Fourier transform $\cF\colon \cS(\R)\to \cS(\R)$ is a topological isomorphism from $\cS(\R)$ onto itself, 
that
can be extended in the usual way to $\cS'(\R)$
, i.e. $\cF(T)(f) :=
\langle T, \hat{f}\rangle$ for every $f \in \cS(\R)$ and $ T\in \cS'(\R)$. Furthemore, the Fourier transfom $\cF$ is a topological isomorphism from the space
$\cO_C'(\R)$  onto the space $\cO_M(\R)$. 

We point out that if $T\in \cO'_C(\R)$ and $f\in\cS(\R)$, then the convolution $T\star f$ satisfies the equality
\begin{equation}\label{eq.FT}
\cF(T\star f)=\cF(T)\hat{f}.
\end{equation}
On the other hand, if $T\in \cO'_C(\R)$ and $S\in \cS'(\R)$, then the convolution $T\star S$ satisfies the equality
\begin{equation}\label{eq.FTD}
\cF(T\star S)=\cF(T)\cF(S).
\end{equation}
In particular,  
applying \eqref{eq.FT} we get  for every $f\in\cS(\R)$ that
\begin{equation}\label{eq.FTC}
\cF(C_T(f))=\cF(T\star f)=\cF(T)\hat{f}.
\end{equation}
While, applying \eqref{eq.FTD} we get for every $S\in \cO'_C(\R)$ that 
\begin{equation}\label{eq.FTDN}
\cF(\cC_T(S))=\cF(T\star S)=\cF(T)\cF(S).
\end{equation}
If we consider the multiplication operator $M_{\cF(T)}\colon \cS(\R)\to \cS(\R)$ and the Fourier transfom $\cF\colon \cS(\R)\to \cS(\R)$ ($M_{\cF(T)}\colon \cO_M(\R)\to \cO_M(\R)$ and $\cF\colon \cO'_C(\R)\to \cO_M(\R)$, resp.), identity \eqref{eq.FTC} (\eqref{eq.FTDN}, resp.) means  that 
\begin{equation}\label{eq.Compo1}
\cF\circ C_T=M_{\cF(T)}\circ \cF\quad (\cF\circ \cC_T= M_{\cF(T)}\circ \cF, {\rm \ resp.}).
\end{equation}
For further properties on the spaces $\cO_M(\R)$ and $\cO_C(\R)$ and the Fourier transform we refer, for instance, to \cite{Ch,Gro,Ho,Sh} (see also \cite{La,LaWe} and the references therein).

\section{Spectra of Multiplication and Convolution operators on $\cS(\R^N)$}

The aim of this section is to  study the spectra of  multiplication operators $M_h$,  for $h\in \cO_M(\R)$, and convolution operators $C_T$, for $T\in \cO'_C(\R)$, when acting on the space $\cS(\R)$.

To this target, we begin by stating and proving some auxiliary results.

\begin{lem}\label{L.ausF} Consider the function
	\begin{equation}\label{eq.funzione}
	f(x)=\frac{1}{1+x},\quad x\not=-1.
	\end{equation}
	Then $f\in C^\infty(\R\setminus\{-1\})$ and
	\begin{equation}\label{eq.derivatef}
	f^{(n)}(x)=(-1)^n\frac{n!}{(1+x)^{n+1}},\quad n\in\N,\ x\not=-1.
	\end{equation}
\end{lem}

\begin{proof} The proof is given by induction. For $n=1$ we have
	\[
	f'(x)=(-1)\frac{1}{(1+x)^2},\quad x\not=-1.
	\]
	Suppose that \eqref{eq.derivatef} is valid for some $n>1$. Then
	\[
	f^{(n+1)}(x)=(-1)^nn!\frac{-(n+1)(1+x)^n}{(1+x)^{2n+2}}=(-1)^{n+1}\frac{(n+1)!}{(1+x)^{n+2}},\quad x\not=-1.
	\]
	So, the proof is complete.
\end{proof}

\begin{lem}\label{eq.ComMul} Let $h\in\cO_M(\R)$. If $(-1)\not\in\overline{{\rm Im}\, h}$, then $f\circ h\in \cO_M(\R)$, where $f$ is the function given in \eqref{eq.funzione}.
\end{lem}

\begin{proof} According to the Fa\`a Bruno formula, we have for every $x\in\R$ and $n\in\N_0$ that 
	\[
	(f\circ h)^{(n)}(x)=\sum\frac{n!}{k_1!k_2!\ldots k_n!}f^{(k)}(h(x))\left(\frac{h'(x)}{1!}\right)^{k_1}\left(\frac{h''(x)}{2!}\right)^{k_2}\ldots \left(\frac{h^{(n)}(x)}{n!}\right)^{k_n},
	\]
	where the sum is extended over all $(k_1,k_2,\ldots,k_n)\in \N_0^n$ such that $k_1+2k_2+\ldots+ nk_n=n$ and $k_1+k_2+\ldots+k_n=k$ (hence, $k_1+k_2+\ldots+k_n\leq n$). Since $(-1)\not\in\overline{{\rm Im}h}$, it clearly follows that $f\circ h\in C^\infty(\R)$. So, it remains to show that $f\circ h\in \cO_M(\R)$. To this end, let choose $d>0$ such that $d<\min\{1,d(-1,\overline{{\rm Im}\,h})\}$ and fix $l\in\N$. Since $h\in \cO_M(\R)$ there exist $C\geq 1$ and $j\in\N$ such that 
	\[
	|h^{(i)}(x)|\leq C(1+x^2)^j,\quad x\in \R,\ i=0,1,\ldots,l.
	\]
	So, we have for every $x\in\R$ and $n=0,1,\ldots,l$ that 
	\begin{align*}
	|(f\circ h)^{(n)}(x)|&\leq \sum\frac{n!}{k_1!k_2!\ldots k_n!}|f^{(k)}(h(x))|\frac{C^k(1+x^2)^{jk}}{(1!)^{k_1}(2!)^{k_2}\ldots (n!)^{k_n}}\\
	&=\sum\frac{n!}{k_1!k_2!\ldots k_n!}\frac{k!}{|1+h(x)|^{k+1}}\frac{C^k(1+x^2)^{jk}}{(1!)^{k_1}(2!)^{k_2}\ldots (n!)^{k_n}}\\
	&\leq \frac{C^n}{d^{n+1}}(1+x^2)^{jn}\sum\frac{n!k!}{k_1!k_2!\ldots k_n!}\frac{1}{(1!)^{k_1}(2!)^{k_2}\ldots (n!)^{k_n}}\\
	&\leq \frac{C^lD}{d^{l+1}}(1+x^2)^{jl},
	\end{align*}
where $D:=\max\{\sum\frac{n!k!}{k_1!k_2!\ldots k_n!}\frac{1}{(1!)^{k_1}(2!)^{k_2}\ldots (n!)^{k_n}}\colon n=0,\ldots,l\}<+\infty$. Since $l\in\N$ is arbitrary, this implies that  $f\circ h\in \cO_M(\R)$.
	\end{proof}

\begin{prop}\label{L.Zero} Let $h\in\cO_M(\R)$. Then the following properties are equivalent.
	\begin{enumerate}
	 \item $M_h\colon \cS(\R)\to \cS(\R)$ is surjective.
	 \item $M_h\colon \cO_M(\R)\to \cO_M(\R)$ is surjective.
	 \item $0\not\in {\rm Im}\, h$ and $\frac{1}{h}\in \cO_M(\R)$.
	 \item  There exist $j\in\N$ and $c>0$ such that $|h(x)|\geq \frac{c}{(1+x^2)^j}$ for every $x\in\R$.
	 \end{enumerate}
\end{prop}

\begin{proof} (3)$\Rightarrow$(1) and  (3)$\Rightarrow$(2) are obviuos.
	
(1)$\Rightarrow$(3).	
%
	Since $M_h$ is surjective, its range ${\rm Im}\, M_h=\cS(\R)$ is clearly a closed subspace of $\cS(\R)$.  Thus, $M_h$ is also injective by \cite{BoFrJo}. So, $M_h$ is a topological isomorphism from  $\cS(\R)$ onto itself. 
	
	Suppose that $0\in {\rm Im}\,h$, i.e., that there exists $x_0\in\R$ such that $h(x_0)=0$. Now, we fix a function  $g\in \cD(\R)$ such that $g(x_0)=1$. Since $M_h$ is surjective, there exists $f\in \cS(\R)$ such that $M_hf=g$. Accordingly,   $h(x)f(x)=g(x)$ for every $x\in\R$ and hence, $0=h(x_0)f(x_0)=g(x_0)=1$. This is a contradiction.
	
	Since $0\not\in {\rm Im}\,h$, the function $\frac{1}{h}\in C^\infty(\R)$. Since $M_h$ is bijective (actually, it suffices the surjectivity), we  have also that
	$\frac{1}{h}g\in \cS(\R)$ for every $g\in\cS(\R)$. This means that  $\frac{1}{h}\in \cO_M(\R)$.

	(2)$\Rightarrow$(3). Since $M_h$ is surjective and the function $\textbf{1}(x):=1$ for $x\in\R$ belongs to $\cO_M(\R)$, there exists $f\in \cO_M(\R)$ such that $M_hf=\textbf{1}$, i.e., $f(x)h(x)=1$ for every $x\in\R$. This necessarily implies that $0\not\in {\rm Im}\, h$ and $f=\frac{1}{h}$. Accordingly, $\frac{1}{h}\in\cO_M(\R)$.
	
	
	(3)$\Rightarrow$(4). By  assumption the function $\frac{1}{h}\in \cO_M(\R)$. Therefore, for $l=0$ there exist $C>0$ and $j\in\N$ such that for every $x\in\R$
	\[
	\left|\frac{1}{h(x)}\right|\leq C(1+x^2)^j.
	\]
	Accordingly, we get for every $x\in\R$ that
	\[
	|h(x)|\geq \frac{1}{C}\frac{1}{(1+x^2)^j}.
	\]
	
	(4)$\Rightarrow$(3). Since $|h(x)|\geq \frac{c}{(1+x^2)^j}$ for every $x\in\R$ and some $c>0$ and $j\in\N$, necessarily  $0\not\in {\rm Im}\,h$. So,  it remains to show that $\frac{1}{h}\in \cO_M(\R)$. To this end, we observe that  $g(x):=f(h(x)-1)=\frac{1}{h(x)}$ for $x\in\R$, being $f$ the function given in \eqref{eq.funzione}. Hence, according to the  Fa\`a Bruno formula and Lemma \ref{L.ausF}, we have for every $n\in\N$ and $x\in\R$ that 
	\begin{align*}
	g^{(n)}(x)&=\sum\frac{n!}{k_1!k_2!\ldots k_n!}f^{(k)}(h(x)-1)\left(\frac{h'(x)}{1!}\right)^{k_1}\left(\frac{h''(x)}{2!}\right)^{k_2}\ldots \left(\frac{h^{(n)}(x)}{n!}\right)^{k_n}\\
	&=\sum\frac{n!}{k_1!k_2!\ldots k_n!}\frac{(-1)^kk!}{(h(x))^{k+1}}\left(\frac{h'(x)}{1!}\right)^{k_1}\left(\frac{h''(x)}{2!}\right)^{k_2}\ldots \left(\frac{h^{(n)}(x)}{n!}\right)^{k_n}
	\end{align*}
	where the sum is extended over all $(k_1,k_2,\ldots,k_n)\in \N_0^n$ such that $k_1+2k_2+\ldots+ nk_n=n$ and $k_1+k_2+\ldots+k_n=k$ (hence, $k_1+k_2+\ldots+k_n\leq n$). But, $h\in\cO_M(\R)$. So, for a fixed $l\in\N$ there exist $C>0$ and $s\in\N$ such that 
	\[
	|h^{(i)}(x)|\leq C(1+x^2)^s,\quad x\in\R,\ i=0,1,\ldots, l.
	\]
	Therefore, it follows for every $x\in\R$ and $n=0,1,\ldots, l$ that 
	\begin{align*}
	|g^{(n)}(x)|&\leq \sum\frac{n!k!}{k_1!k_2!\ldots k_n!}c^{-(k+1)}(1+x^2)^{j(k+1)}\frac{C^{k}(1+x^2)^{sk}}{(1!)^{k_1}(2!)^{k_2}\ldots (n!)^{k_n}}\\
	&\leq \left(\frac{C}{c}\right)^{n+1}(1+x^2)^{n(s+j)+1}\sum\frac{n!k!}{k_1!k_2!\ldots k_n!}\frac{1}{(1!)^{k_1}(2!)^{k_2}\ldots (n!)^{k_n}}\\
	&\leq \left(\frac{C}{c}\right)^{l+1}D(1+x^2)^{l(s+j)+1},
	\end{align*} 
	where $D:=\max\{\sum\frac{n!k!}{k_1!k_2!\ldots k_n!}\frac{1}{(1!)^{k_1}(2!)^{k_2}\ldots (n!)^{k_n}}\colon n=0,1,\ldots, l\}<+\infty$ (here, without loss of generality, we have assumed that $C>1$ and $0<c<1$). Since $l\in\N$ is arbitrary, this means that $\frac{1}{h}=g\in\cO_M(\R)$. 
	\end{proof}

\begin{rem}\label{R.onto} We observe that the condition $0\not\in {\rm Im}\,h$ does not imply that $M_h$ is surjective. Indeed, consider the function $h(x)=e^{-x^2}$ for $x\in\R$. Then $0\not\in {\rm Im}\,h=]0,1]$. But,  $M_h\colon \cS(\R)\to \cS(\R)$ is not surjective. Otherwise,  by Lemma \ref{L.Zero} the function $\frac{1}{h(x)}=e^{x^2}$, for $x\in\R$, should  belong to  $\cO_M(\R)$.  This is false.	
\end{rem}

We are ready to study the spectra of the multiplication operators acting  on $\cS(\R)$. 

\begin{thm}\label{T.Spectra} Let $h\in\cO_M(\R)$. Then the spectra of the multiplication operator $M_h$ acting either on $\cS(\R)$ or on $\cO_M(\R)$ are given by: 
	\begin{equation}\label{eq.Risolvente}
(\overline{{\rm Im}\,h})^c\subseteq	\rho(M_h)\subseteq ({\rm Im}\, h)^c,\ \rho^*(M_h)=(\overline{{\rm Im}\,h})^c,
\end{equation}
\begin{equation}\label{eq.Spectrum}
{\rm Im}\,h \subseteq\sigma(M_h)\subseteq \overline{{\rm Im}\,h,} \ \sigma^*(M_h)=\overline{{\rm Im}\,h},
	\end{equation}
	\begin{equation}\label{eq.PointSpectrum}
	\sigma_p(M_h)=\{\lambda\in\C\colon {\rm } {h^{-1}(\lambda)}\ {\rm has\ a \ non \ empty \ interior}\}\subseteq {\rm Im}\,h.
	\end{equation}
	\end{thm}

\begin{proof} We first consider the case $M_h\colon \cS(\R)\to \cS(\R)$. 
	
	Let $\lambda\not\in \overline{{\rm Im}\,h}$ and let $0<2d<\min\{1,d(\lambda,\overline{{\rm Im}\,h})\}$. Then $d(B(\lambda,d), \overline{{\rm Im}\,h}) \geq d$, thereby implying for every $\mu\in B(\lambda,d)$ that $(-1)\not\in \overline{{\rm Im}\,(\mu-h-1)}$. So, by Lemma \ref{eq.ComMul} for every $\mu\in B(\lambda,d)$ the function $\frac{1}{\mu-h}\in \cO_M(\R)$ and hence, $M_{\frac{1}{\mu-h}}\in \cL(S(\R))$. Since $(\mu I -M_h)M_{\frac{1}{\mu-h}}=I=M_{\frac{1}{\mu-h}}(\mu I-M_h)$ for  $\mu\in B(\lambda,d)$, it follows that $B(\lambda,d)\subseteq \rho(M_h)$. It remains to show that the set $\{M_{\frac{1}{\mu-h}}\colon \mu\in B(\lambda,\mu)\}$ is equicontinuous in $\cL(\cS(\R))$. To do this, we recall that   $h\in\cO_M(\R)$ and hence,   for a fixed $l\in\N$  there exist $C\geq 1$ and $j\in\N$ such that
	\[
	|h^{(i)}(x)|\leq C(1+x^2)^j, \quad x\in\R,\ i=0,1,\ldots,l.
	\]
Now, taking into  account of $\frac{1}{\mu-h}=\frac{1}{1+(\mu-h-1)}=f\circ (\mu-h-1)$ for $\mu\in B(\lambda,d)$, where $f(x)=\frac{1}{1+x}$ for $x\not=-1$,  and   proceeding as in the proof of Lemma \ref{eq.ComMul}, we obtain for every $x\in\R$, $n=0,1,\ldots, l$ and $\mu\in B(\lambda,d)$   that 
\begin{equation}\label{eq.DerivateTutte}
\left(\frac{1}{\mu-h(x)}\right)^{(n)}\leq \frac{C^lD}{d^{l+1}}(1+x^2)^{jl},
\end{equation}
where $D>0$ is a constant independent of $\mu$ and $h$. Applying  \eqref{eq.DerivateTutte}, we obtain for every $\mu\in B(\lambda,d)$,   $g\in\cS(\R)$ and $n=0,\ldots,l$, $x\in\R$ that 
\begin{align*}
(1+x^2)^l\left|\left(M_{\frac{1}{\mu-h}}g(x)\right)^{(n)}\right|&\leq \sum_{i=0}^n\binom{n}{i}\left|\left(\frac{1}{\mu-h(x)}\right)^{(n-i)}\right|(1+x^2)^l|g^{(i)}(x)|\\
&\leq \sum_{i=0}^n\binom{n}{i}\frac{C^lD}{d^{l+1}}(1+x^2)^{(j+1)l}|g^{(i)}(x)|\leq \frac{(2C)^lD}{d^{l+1}}||g||_{(j+1)l},
\end{align*}
and hence,
\begin{equation}\label{eq.StimaR}
\left|\left| M_{\frac{1}{\mu-h}}g\right|\right|_l\leq \frac{(2C)^lD}{d^{l+1}}||g||_{(j+1)l}.
\end{equation}
Since $l\in\N$ is arbitrary, from  \eqref{eq.StimaR} it   follows that the set $\{M_{\frac{1}{\mu-h}}\colon \mu\in B(\lambda,\mu)\}$ is equicontinuous in $\cL(\cS(\R))$. So, we have established that $(\overline{{\rm Im}\,h})^c\subseteq \rho^*(M_h) \subseteq \rho(M_h)$. 

Let $\lambda\in \rho(M_h)$. Then the operator $M_{\lambda-h}$ is a topological isomorphism from $\cS(\R)$ onto itself and hence, it is surjective. By Proposition \ref{L.Zero} it follows that $0\not\in {\rm Im}\, (\lambda-h)$ and hence $\lambda \not \in {\rm Im}\, h$. So, we have established that $\rho(M_h)\subseteq ({\rm Im}\,h)^c$.

Now, from $(\overline{{\rm Im}\,h})^c\subseteq \rho^*(M_h)\subseteq  \rho(M_h)\subseteq ({\rm Im}\,h)^c$ if follows that ${\rm Im}\,h\subseteq \sigma(M_h)\subseteq\sigma^*(M_h)\subseteq  \overline{{\rm Im}\, h}$. Since $\sigma^*(M_h)$ is closed, this yields  that $\sigma^*(M_h)=  \overline{{\rm Im}\, h}$ and hence, $\rho^*(M_h)=(\overline{{\rm Im}\, h})^c$.

Finally, suppose that $(\lambda I-M_h)g=0$ for some $\lambda\in\C$ and $g\in \cS(\R)$ with $g\not=0$. Since $g\not=0$, there exists an open subset $U$ of $\R$ such that $g(x)\not =0$ whenever $x\in U$. So, $h(x)=\lambda$ for every $x\in U$, thereby implying that $h^{-1}(\lambda)$ has a non empty interior. Conversely, if $h^{-1}(\lambda)$ has a non empty interior, then any function $g\in\cD(U)$, $g\not=0$, with $U$ the interior set of $h^{-1}(\lambda)$, clearly satisfies  $(\lambda I-M_h)g=0$, i.e., $\lambda\in \sigma_p(M_h)$. This completes the proof.

The proof for the case $M_h\colon \cO_M(\R)\to\cO_M(\R)$ follows by arguing as above with some obviuos changes. We only prove that $(\overline{{\rm Im}\, h})^c\su\rho^*(M_h)$. So, $\lambda\not\in \overline{{\rm Im}\, h}$ and let $0<2d<\min\{1,d(\lambda, \overline{{\rm Im}\, h})\}$. Then for a fixed $l\in\N$, the same arguments above show that inequality \eqref{eq.DerivateTutte} is valid for every $x\in\R$, $n=0,\ldots, l$ and $\mu\in B(\lambda,d)$. Applying \eqref{eq.DerivateTutte}, we obtain for every $\mu\in B(\lambda,d)$, $g\in \cO_M(\R)$, $v\in \cS(\R)$ and $n=0,\ldots, l$, $x\in\R$ that
\begin{align*}
|v(x)|\left|\left(M_{\frac{1}{\mu-h}}g(x)\right)^{(n)}\right|&\leq \sum_{i=0}^n\binom{n}{i}\left|\left(\frac{1}{\mu-h(x)}\right)^{(n-i)}\right||v(x)g^{(i)}(x)|\\
&\leq \sum_{i=0}^n\binom{n}{i}\frac{C^lD}{d^{l+1}}(1+x^2)^{jl}|v(x)g^{(i)}(x)|\leq \frac{(2C)^lD}{d^{l+1}}p_{l,w}(g),
\end{align*}
where $w(x):=(1+x^2)^{jl}v(x)$ for $x\in\R$ and hence, $w\in \cS(\R)$. Therefore, we get for every $\mu\in B(\lambda,d)$ and $g\in \cO_M(\R)$ that 
\[
p_{l,v}\left(M_{\frac{1}{\mu-h}}g\right)\leq \frac{(2C)^lD}{d^{l+1}}p_{l,w}(g).
\]
Since $l\in\N$ and $v\in\cS(\R)$ are arbitrary, this means that the set $\{M_{\frac{1}{\mu-h}}\colon \mu\in B(\lambda,d)\}$ is equicontinuous in $\cL(\cO_M(\R))$.
\end{proof}

\begin{rem} The description of $\rho(M_h)$ given in \eqref{eq.Risolvente} cannot be improved. Indeed, for the function $h(x):=e^{-x^2}$ for $x\in\R$ we have $0\not\in {\rm Im}\,h$ as $ {\rm Im}\,h=]0,1]$. But, $0\not\in \rho(M_h)$, see Remark \ref{R.onto}. So, by Theorem \ref{T.Spectra} $\rho^*(M_h)=\rho(M_h)=([0,1])^c$.
	
	 On the other hand, for the function $k(x)=\frac{1}{1+x^2}$ for $x\in\R$ we have that $0\not\in  {\rm Im}\,k=]0,1]$ and $\frac{1}{k(x)}=1+x^2\in \cO_M(\R)$. By Proposition \ref{L.Zero} the operator $M_k\colon \cS(\R)\to \cS(\R)$ is surjective. According to \cite{BoFrJo},   $M_k\colon \cS(\R)\to \cS(\R)$ is also injective. Thus, $M_k$ is a topological isomorphism from $\cS(\R)$ onto itself and so $0\in\rho(M_k)$. So, by Theorem \ref{T.Spectra} $\rho^*(M_h)=([0,1])^c\subsetneqq\rho(M_h)=(]0,1])^c$. Moreover, $\sigma(M_k)=]0,1]\subsetneqq \sigma^*(M_k)=[0,1]$.
\end{rem}

As a consequence of Theorem \ref{T.Spectra} we  can determine the spectra of the convolution operators $C_T$ acting on $\cS(\R)$ via the properties of the Fourier transform.

\begin{thm}\label{T.SpectraC} Let $T\in \cO_C'(\R)$. Then the spectra of the convolution operator $C_T$ acting either on $\cS(\R)$ or on $\cO_C(\R)$ are given by:
		\begin{equation}\label{eq.RisolventeC}
	(\overline{{\rm Im}\,\cF(T)})^c\subseteq	\rho(C_T)\subseteq ({\rm Im}\, \cF(T))^c,\ \rho^*(C_T)=(\overline{{\rm Im}\,\cF(T)})^c,
	\end{equation}
	\begin{equation}\label{eq.SpectrumC}
	{\rm Im}\,\cF(T) \subseteq\sigma(C_T)\subseteq \overline{{\rm Im}\,\cF(T),} \ \sigma^*(C_T)=\overline{{\rm Im}\,\cF(T)}.
	\end{equation}
	Moreover, the point spectrum of  the convolution operator $C_T$ acting  on $\cS(\R)$ is given by:
	\begin{equation}\label{eq.PointSpectrumC}
	\sigma_p(C_T)=\{\lambda\in\C\colon {\rm } {(\cF(T))^{-1}(\lambda)}\ {\rm has\ a \ non \ empty \ interior}\}\subseteq {\rm Im}\,\cF(T).
	\end{equation}
\end{thm} 

\begin{proof} In the case $C_T\colon \cS(\R)\to \cS(\R)$, the proof immediately follows from Theorem \ref{T.Spectra}, after having observed that $\cF\circ C_T=M_{\cF(T)}\circ \cF$ (by \eqref{eq.Compo1} with $n=1$) where $\cF(T)\in \cO_M(\R)$, and that the Fourier transfom $\cF$ is a topological isomorphism from $\cS(\R)$ onto itself.
	
	In the case $C_T\colon \cO_C(\R)\to \cO_C(\R)$, we first observe that Theorem \ref{T.Spectra} and $\cF\circ \cC_T=M_{\cF(T)}\circ \cF$, where $\cF\colon \cO_C'(\R)\to \cO_M(\R)$ is a topological isomorphism onto, imply that \eqref{eq.RisolventeC} and \eqref{eq.SpectrumC} ( and \eqref{eq.PointSpectrumC}) are true with $\cC_T$ instead of $C_T$. Now, the result follows by  duality, being $C_T=\cC_T'$.
\end{proof}

\section{Ergodic properties  of Multiplication and Convolution  operators on $\cS(\R)$}

The aim of this section is to investigate the ergodic  properties of the multiplication operators $M_h$, for $h\in\cO_M(\R)$, and the convolution operators $C_T$, for $T\in \cO_C'(\R)$, when acting on the space $\cS(\R)$. 

So, we first observe that if $f\in\cS(\R)$ (or $f\in\cO_M(\R)$) and $n\in\N$, 
\[
M_h^nf(x)=h^n(x)f(x),\quad x\in\R.
\]
Therefore, if $f\in\cS(\R)$ (or $f\in\cO_M(\R)$) and $n\in\N$, the $n$th-Ces\`aro mean of $M_h$ is given by  
\[
(M_h)_{[n]}f(x)=\frac{1}{n}\sum_{m=1}^nM^m_hf(x)=\frac{1}{n}\sum_{m=1}^nh^m(x)f(x)=\frac{f(x)}{n}\sum_{m=1}^nh^m(x),\quad x\in\R.
\]
If we set $h_{[n]}:=\frac{1}{n}\sum_{m=1}^nh^m$ for $n\in\N$, then $(M_h)_{[n]}f=h_{[n]}f$ for any $f\in \cS(\R)$ (or $f\in \cO_M(\R)$).

\begin{thm}\label{T.PowerB} Let $h\in \cO_{M}(\R)$. Then the following properties are equivalent.
	\begin{enumerate}
		\item[(1)] $M_h$ is power bounded on $\cS(\R^N)$.
		\item[(1)'] $\cM_h$ is power bounded on $\cS'(\R^N)$.
		\item[(2)] $M_h$ is power bounded on $\cO_M(\R^N)$.
		\item[(2)'] $\cM_h$ is power bounded on $\cO'_M(\R^N)$.
		\item[(3)] The sequence $\{h^n\}_{n\in\N}$ is bounded in $\cO_{M}(\R)$.
	\end{enumerate}
\end{thm}
	
	\begin{proof} (1)$\Leftrightarrow$(1)' follows from \cite[Lemma 2]{ABR-2}, after having observed that $\cM_h=M_h'$ and $\cS(\R)$ is a reflexive lcHs.
		
		(2)$\Leftrightarrow$(2)' follows by the same arguments above.
		
		(1)$\Rightarrow$(2). Since $\cS(\R)$ is continuously included in $C^\infty(\R)$, the assumption implies that  $\{M_h^nf\}_{n\in\N}$ is a bounded subset of $C^\infty(\R)$ whenever $f\in \cS(\R)$. This necessarily yields that $\{h^n\}_{n\in\N}$ is a bounded sequence in $C^\infty(\R)$. Indeed, fixed a compact subset $K$ of $\R$ and choosen $f\in \cD(\R)$ such that $f(x)=1$ for every $x\in K$, we obtain for every  $n\in\N$ and $m\in \N$ that 
		\[
		\sup_{x\in K}\sup_{0\leq i\leq m}|(h^{n}(x))^{(i)}|=\sup_{x\in K}\sup_{0\leq i\leq m}|(h^{n}(x)f(x))^{(i)}|\leq C
		\]
	where $C:=\sup_{n\in\N}\sup_{x\in K}\sup_{0\leq i\leq m}|(h^{n}(x)f(x))^{(i)}|<+\infty$ as $f\in \cS(\R)$. Since $K$ is an arbitrary compact subset of $\R$, this means that $\{h^n\}_{n\in\N}$ is a bounded sequence of $C^\infty(\R)$. 
	
	The sequence $\{h^n\}_{n\in\N}$ is also  bounded in  $\cO_M(\R)$. Otherwise,  there exists $m\in\N$ such that the sequence $\{h^n\}_{n\in\N}$ is not bounded in $\cup_{k=1}^\infty\{f\in C^\infty(\R)\colon |f|_{m,k}=\sup_{x\in\R}\sup_{0\leq i\leq m}(1+x^2)^{-k}|f^{(i)}(x)|<\infty\}$. So, for every $k\in\N$ we have 
	\[
	\sup_{n\in\N}|h^n|_{m,k}=\infty.
	\]
	Consequently,
	taking into account that $\{h^n\}_{n\in\N}$ is a bounded sequence of $C^\infty(\R)$, we can find a sequence $\{x_k\}_{k\in\N}\su\R$ with $|x_k|>|x_{k-1}|+2$, for $k\geq 2$, an increasing sequence $\{n_k\}_{k\in\N}$ of positive integers and a sequence $\{i_k\}_{k\in\N}\su \{0,1,\ldots,m\}$ such that 
	\begin{equation}\label{eq.CCP}
	(1+x^2_k)^{-k}|(h^{n_k})^{(i_k)}(x_k)|>k, \quad k\in\N.
	\end{equation}
	Since the set $\{0,1,\ldots,m\}$ is finite, we can assume without loss of generality (indeed, it suffices to pass to a subsequence) that \eqref{eq.CCP} is valid for every $k\in\N$ with the same index $i$ in $\{0,1,\ldots,m\}$. Now, let $\rho\in\cD(\R)$ such that ${\rm supp}\, \rho \su ]-2,2[$ and $\rho(x)=1$ for $x\in [-1,1] $. Then  the function
	\[
	\rho(x)=\sum_{k=1}^\infty\frac{\rho(x-x_k)}{(1+x^2_k)^k},\quad x\in\R, 
	\]
	belongs to $\cS(\R)$, see, f.i., \cite[Proposition 4, Chap. 4 \S 11]{Ho}. Since $M_h$ is power bounded on $\cS(\R)$, the sequence $\{M_h^n\rho\}_{n\in\N}=\{h^n\rho\}_{n\in\N}$ is bounded in $\cS(\R)$. This implies that 
	\begin{equation}\label{eq.BPP}
	\sup_{n\in\N}\sup_{x\in\R}|(h^n(x)\rho(x))^{(i)}|=:C<\infty.
	\end{equation}
	But, by \eqref{eq.CCP} we get
	\[
	\sup_{k\in\N}\sup_{|x-x_k|<1}|(h^{n_k}(x)\rho(x))^{(i)}|=\sup_{k\in\N}\sup_{|x-x_k|<1}\left|\frac{(h^{n_k}(x))^{(i)}}{(1+x_k^2)^k}\right|\geq \sup_{k\in\N}\left|\frac{(h^{n_k})^{(i)}(x_k)}{(1+x_k^2)^k}\right|=+\infty.
	\] 
	This is a contradiction with \eqref{eq.BPP}.
	
	Since $\{h^n\}_{n\in\N}$ is a bounded sequence in $\cO_M(\R)$, for each $v\in \cS(\R)$ and $m\in\N$ there exists $C_{m,v}>0$ such that 
	\begin{equation}\label{eq.BBPB}
	\sup_{n\in\N}p_{m,v}(h^n)\leq C_{m,v}.
	\end{equation}
	 This together with \eqref{eq.Prod} clearly implies that $\{M_h^n\}_{n\in\N}$ is an equicontinuous set in $\cL(\cO_M(\R))$, i.e., $M_h$ is power bounded when it acting on $\cO_M(\R)$. Indeed, fixed $v\in\cS(\R)$ and $m\in\N$ and choosen $v_1,v_2\in\cS(\R)$ and $m'\in\N$ as in \eqref{eq.Prod}, we have for every $n\in\N$ and $f\in \cO_M(\R)$ that 
	 \[
	 p_{m,v}(M_h^nf)=p_{m,v}(h^nf)\leq p_{m',v_1}(h^n)p_{m',v_2}(f)\leq C_{m',v_1}p_{m',v_2}(f).
	 \]
	
	(2)$\Rightarrow$(3). The operator $M_h$ is power bounded on $\cO_M(\R)$. Accordingly, for the constant function $\mathbf{1}$ which belongs to $\cO_M(\R)$, the sequence $\{M^n\mathbf{1}\}_{n\in\N}=\{h^n\}_{n\in\N}$ is  bounded in $\cO_M(\R)$.
	
	(3)$\Rightarrow$(1). Let $m\in\N$ be fixed. Since $\{h^n\}_{n\in\N}$ is a bounded sequence in $\cO_M(\R)$ and $\cO_M(\R)$ is the projective limit of the regular (LB)-spaces $\{\cup_{k=1}^\infty\{f\in C^\infty(\R)\colon |f|_{l,k}=\sup_{x\in\R}\sup_{0\leq i\leq l}(1+x^2)^{-k}|f^{(i)}(x)|<\infty\}\}_{l\in\N}$ , it is also a bounded sequence in  $\cup_{k=1}^\infty \{f\in C^\infty(\R)\colon |f|_{m,k}=\sup_{x\in\R}\sup_{0\leq i\leq m}(1+x^2)^{-k}|f^{(i)}(x)|<\infty\}$ and hence, there exists $k\in\N$ such that
	\[
	\sup_{n\in\N}|h^n|_{m,k}=:C<\infty.
	\]
	Therefore, we obtain for every $n\in\N$ and $f\in\cS(\R)$ that 
	\begin{align*}
	||M_h^nf||_m&=\sup_{x\in\R}\sup_{0\leq i\leq m}(1+x^2)^m|(h^n(x)f(x))^{(i)}|\\
	&\leq \sup_{x\in\R}\sup_{0\leq i\leq m}(1+x^2)^m\sum_{j=0}^i\binom{i}{j}|(h^n(x))^{(j)}f^{(i-j)}(x)|\\
	&\leq \sup_{x\in\R}\sup_{0\leq i\leq m}(1+x^2)^{m+k}\sum_{j=0}^i\binom{i}{j}(1+x^2)^{-k}|(h^n(x))^{(j)}||f^{(i-j)}(x)|\\
	&\leq 2^mC||f||_{m+k}.
	\end{align*}
	This means that $M_h$ is power bounded on $\cS(\R)$.
\end{proof}

\begin{thm}\label{T.MeanERg} Let $h\in \cO_{M}(\R)$. Then the following properties are equivalent.
	\begin{enumerate}
		\item[(1)] $M_h$ is (uniformly) mean ergodic on $\cS(\R^N)$.
		\item[(1)'] $\cM_h$ is (uniformly) mean ergodic on $\cS'(\R^N)$.
		\item[(2)] $M_h$ is  (uniformly) mean ergodic on $\cO_M(\R^N)$.
		\item[(2)'] $\cM_h$ is  (uniformly) mean ergodic on $\cO'_M(\R^N)$.
		\item[(3)] The sequence $\{h_{[n]}\}_{n\in\N}$ converges to some  $g$ in $\cO_{M}(\R)$.
	\end{enumerate}
	\end{thm}

\begin{proof} (1)$\Leftrightarrow$(1)' follows from \cite[Lemma 2.1]{ABR-1}, after having observed that $\cM_h=M_h'$ and $\cS(\R)$ is a reflexive lcHs (and hence, $\cS'(\R)$ is a barrelled reflexive lcHs).
	
(2)$\Leftrightarrow$(2)' follows as above.	
	
	We now establish the equivalence between properties (2) and (3).
	
	(2)$\Rightarrow$(3). The operator $M_h$ is mean ergodic on $\cO_M(\R^N)$. Accordingly, there exists $P\in \cL(\cO_M(\R))$ such that $(M_h)_{[n]}\to P$ in $\cL_s(\cO_M(\R))$ as $n\to\infty$. Since the constant function $\mathbf{1}$ belongs to $\cO_{M}(\R)$, it follows that  $M_{[n]}\mathbf{1}=h_{[n]}\to P\mathbf{1}=:g$ in $\cO_M(\R)$.
	
	(3)$\Rightarrow$(2). 
	Fixed $m\in\N$ and $v\in\cS(\R)$ and choosen $v_1,v_2\in \cS(\R)$ and $m'\in\N$ as in \eqref{eq.Prod}, we obtain for every $f\in\cO_M(\R)$ and $n\in\N$ that
	\begin{align} \label{eq.CB}
	 p_{m,v}((M_h)_{[n]}f-fg)&=p_{m,v}(h_{[n]}f-fg)=p_{m,v}((h_{[n]}-g)f)\nonumber\\
	 &\leq p_{m',v_1}(h_{[n]}-g)p_{m',v_2}(f).
	\end{align}
	Since $h_{[n]}\to g$ in $\cO_{M}(\R)$ and $m\in\N$, $v\in\cS(\R)$ and $f\in\cO_M(\R)$ are arbitrary, from \eqref{eq.CB} it follows that $(M_h)_{[n]}f\to fg$ in $\cO_M(\R)$ as $n\to\infty$. This means that $M_h$ is mean ergodic on $\cO_M(\R)$ and hence, uniformly mean ergodic, being $\cO_M(\R)$  a Montel lcHs. 
	
Actually, 	 from \eqref{eq.CB} it directly  follows  that $M_h$ is uniformly mean ergodic on $\cO_M(\R)$. Indeed, for a fixed bounded subset $B$ of $\cO_M(\R)$, we obtain by \eqref{eq.CB} for every $n\in\N$ that
\[
\sup_{f\in B}p_{m,v}((M_h)_{[n]}f-fg)\leq p_{m',v_1}(h_{[n]}-g)\sup_{f\in B}p_{m',v_2}(f),
\] 
where $\sup_{f\in B}p_{m',v_2}(f)<\infty$ and $p_{m',v_1}(h_{[n]}-g)\to 0$ as $n\to\infty$.

	(3)$\Rightarrow$(1). Let $f\in \cS(\R)$ and $m\in\N$ be fixed. Then there exists $v\in \cS(\R)$ such that 
	\begin{equation}\label{eq.AulF}
	\sup_{0\leq j\leq m}(1+x^2)^m|f^{(j)}(x)|\leq v(x)
	\end{equation} 
	for every $x\in\R$ (see \cite[Lemma 3.6, p. 127]{Ch}). Thus, we obtain for every $n\in\N$ that
	\begin{align*}
	& ||(M_h)_{[n]}f-gf||_m=\sup_{x\in\R}\sup_{0\leq i\leq m}(1+x^2)^m|[(h_{[n]}(x)-g(x))f(x)]^{(i)}|\\
	& \leq \sup_{x\in\R}\sup_{0\leq i\leq m}(1+x^2)^m\sum_{j=0}^i\binom{i}{j}|(h_{[n]}(x)-g(x))^{(j)}||f^{(j-i)}(x)|\\
	&\leq 2^m\sup_{x\in\R}\sup_{0\leq i\leq m}v(x)|(h_{[n]}(x)-g(x))^{(i)}|=p_{m,v}(h_{[n]}-g).
	\end{align*}
	Since $h_{[n]}\to g$ in $\cO_M(\R)$ and  $m\in\N$ is arbitrary, this implies that $(M_h)_{[n]}f\to gf$ in $\cS(\R)$ as $n\to\infty$. As $f\in\cS(\R)$ is also arbitary, we can  conclude that $M_h$ is mean ergodic when it acting on $\cS(\R)$ and hence, uniformly mean ergodic, being $\cS(\R)$  a Montel Fr\'echet space.

To complete the proof, it remains to show that (1)$\Rightarrow$(2).

(1)$\Rightarrow$(2). Since $\cS(\R)$ is continuously included in $C^\infty(\R)$, the assumption implies that $(M_h)_{[n]}f\to Pf$ in $C^\infty(\R)$ too, as $n\to\infty$, whenever $f\in \cS(\R)$.
This necessarily implies  that $\{h_{[n]}\}_{n\in\N}$ converges to some $g$ in $C^\infty(\R)$. Indeed, fixed a compact subset $K$ of $\R$ and choosen $f\in \cD(\R)$ such that $f(x)=1$ for every $x\in K$, we have for every $l,m,n\in\N$ that 
\begin{align}\label{eq.DD}
&\sup_{x\in K}\sup_{0\leq i\leq l}|(h_{[n]}(x)-h_{[m]}(x))^{(i)}|=\sup_{x\in K}\sup_{0\leq i\leq l}|[(h_{[n]}(x)-h_{[m]}(x))f(x)]^{(i)}|\nonumber\\
&\leq \sup_{x\in K}\sup_{0\leq i\leq l}|(h_{[m]}(x)f(x)-Pf(x))^{(i)}|+\sup_{x\in K}\sup_{0\leq i\leq l}|(h_{[n]}(x)f(x)-Pf(x))^{(i)}|\\
&=:a_m+a_n.\nonumber
\end{align} 
Since $(M_h)_{[n]}f=h_{[n]}f\to Pf$ in $C^\infty(\R)$ as $n\to\infty$, we get that  $a_m+a_n\to 0$ as $m,n\to\infty$. Hence, it follows via \eqref{eq.DD} that  $\sup_{x\in K}\sup_{0\leq i\leq l}|(h_{[n]}(x)-h_{[m]}(x))^{(i)}|\to 0$ as $m,n\to\infty$. So,  as $K$ is an arbitrary compact subset of $\R$, this shows that  
$\{h_{[n]}\}_{n\in\N}$ is a Cauchy sequence in $C^\infty(\R)$. Accordingly, $h_{[n]}\to g$ in $C^\infty(\R)$ as $n\to \infty$.

The facts that $(M_h)_{[n]}f=h_{[n]}f\to Pf$ and $(M_h)_{[n]}f=h_{[n]}f\to gf$ in $C^\infty(\R)$ as $n\to\infty$ for every $f\in \cS(\R)$, imply that $gf=Pf\in \cS(\R)$ for every $f\in \cS(\R)$. Therefore, $g\in \cO_M(\R)$.

The assumption also implies  that $\{h_{[n]}\}_{n\in\N}$ is a bounded sequence of $\cO_M(\R)$. Otherwise,  there exists $m\in\N$ such that the sequence $\{h_{[n]}\}_{n\in\N}$ is not bounded in $\cup_{k=1}^\infty\{f\in C^\infty(\R)\colon |f|_{m,k}=\sup_{x\in\R}\sup_{0\leq i\leq m}(1+x^2)^{-k}|f^{(i)}(x)|<\infty\}$. So, for every $k\in\N$ we have
\[
\sup_{n\in\N}|h_{[n]}|_{m,k}=\infty.
\]
Consequently, taking into account that the sequence $\{h_{[n]}\}_{n\in\N}$ is bounded in $C^\infty(\R)$, we can find a sequence $\{x_k\}_{k\in\N}\su \R$ with $|x_k|>|x_{k-1}|+2$, for $k\geq 2$, an increasing sequence $\{n_k\}_{k\in\N}$ of positive integers and a sequence $\{i_k\}_{k\in\N}\su \{0,1,\ldots,m\}$ such that 
\begin{equation}\label{eq.CC}
(1+x^2_k)^{-k}|h_{[n_k]}^{(i_k)}(x_k)|>k, \quad k\in\N.
\end{equation}
Since the set $\{0,1,\ldots,m\}$ is finite, we can assume without loss of generality (indeed, it suffices to pass to a subsequence) that \eqref{eq.CC} is valid for every $k\in\N$ with the same index $i$ in $\{0,1,\ldots,m\}$. Now, let $\rho\in\cD(\R)$ such that ${\rm supp}\, \rho \su ]-2,2[$ and $\rho(x)=1$ for $x\in [-1,1] $. Then  the function
\[
\rho(x)=\sum_{k=1}^\infty\frac{\rho(x-x_k)}{(1+x^2_k)^k},\quad x\in\R, 
\]
belongs to $\cS(\R)$, see, f.i., \cite[Proposition 4, Chap. 4 \S 11]{Ho}. Since $M_h$ is mean ergodic in  $\cS(\R)$, we have that $(M_h)_{[n]}\rho =h_{[n]}\rho\to P\rho=g\rho$ in $\cS(\R)$ and hence, the sequence $\{h_{[n]}\rho\}_{n\in\N}$ is bounded in $\cS(\R)$. This implies that 
\begin{equation}\label{eq.CCC}
\sup_{n\in\N}\sup_{x\in\R}|(h_{[n]}(x)\rho(x))^{(i)}|=:C<\infty.
\end{equation}
But, by \eqref{eq.CC} we have
\[
 \sup_{k\in\N}\sup_{|x-x_k|<1}|(h_{[n_k]}(x)\rho(x))^{(i)}|=\sup_{k\in\N}\sup_{|x-x_k|<1}\left|\frac{h_{[n_k]}^{(i)}(x)}{(1+x_k^2)^k}\right|\geq \sup_{k\in\N}\left|\frac{h_{[n_k]}^{(i)}(x_k)}{(1+x_k^2)^k}\right| =+\infty.
\]
This is a contradiction with \eqref{eq.CCC}.

So,   $\{h_{[n]}\}_{n\in\N}$ is  a  bounded sequence in $\cO_M(\R)$. Accordingly, for each $v\in \cS(\R)$ and $m\in\N$ there exists $C_{m,v}>0$ such that 
\begin{equation}\label{eq.BB}
\sup_{n\in\N}p_{m,v}(h_{[n]})\leq C_{m,v}.
\end{equation}
This together with \eqref{eq.Prod} clearly implies that $\{(M_h)_{[n]}\}_{n\in\N}$ is an equicontinuous set in $\cL(\cO_M(\R))$. Indeed, fixed $v\in\cS(\R)$ and $m\in\N$ and choosen $v_1,v_2\in\cS(\R)$ and $m'\in\N$ as in \eqref{eq.Prod}, we have for every $n\in\N$ and $f\in \cO_M(\R)$ that 
\[
p_{m,v}((M_h)_{[n]}f)=p_{m,v}(h_{[n]}f)\leq p_{m',v_1}(h_{[n]})p_{m',v_2}(f)\leq C_{m',v_1}p_{m',v_2}(f).
\]
Since $\{(M_h)_{[n]}\}_{n\in\N}$ is an equicontinuous set in $\cL(\cO_M(\R))$ and $(M_h)_{[n]}f\to Pf(=gf)$ in $\cS(\R)$ as $n\to\infty$ (and hence in $\cO_M(\R)$),  and $\cS(\R)$ is a dense subspace of  $\cO_M(\R)$, we deduce  that $(M_h)_{[n]}f\to Pf(=gf)$ in $\cO_M(\R)$ as $n\to\infty$, for every $f\in  \cO_M(\R)$.	So, $M_h$ is mean ergodic on $\cO_M(\R)$ and hence, uniformly mean ergodic, being $\cO_M(\R)$ a Montel lcHs.
\end{proof}

We now point out the following fact which could be useful for applications.

\begin{prop}\label{P.Funzioneh} Let $h\in\cO_M(\R)$ with $h\not={\bf 1}$. If the multiplication operator $M_h\colon \cS(\R)\to\cS(\R)$ is either power bounded or mean ergodic, then $||h||_0=\sup_{x\in\R}|h(x)|\leq 1$ and $h^{-1}(1)$ is an empty subset of $\R$. Moreover, the operator $P=\lim_{n\to\infty}(M_h)_{[n]}=0$  and $\cS(\R)=\overline{{\rm Im}\, M_{1-h}}$.
	\end{prop}

\begin{proof} The assumption implies that $\frac{M_h^n}{n}\to 0$ in $\cL_s(\cS(\R))$ and hence, by Proposition \ref{P.limite},  we have $\sigma(M_h)\su \overline{\D}$. On the other hand, by Theorem \ref{T.Spectra}, ${\rm Im}\,h\su \sigma(M_h)$. So, it follows that  $||h||_0\leq 1$.
	

Suppose that $h^{-1}(1)\not=\emptyset$. Since $||h||_0\leq 1$ and 
\[
h_{[n]}(x)=\frac{1}{n}\sum_{m=1}^nh^m(x)=\left\{
\begin{array}{cc}
\frac{h(x)}{n}\frac{1-h^n(x)}{1-h(x)},\ & {\rm if } \ h(x)\not=1,\\
1, \ & {\rm if } \ h(x)=1,\\
\end{array}\right.
\]
for every $n\in\N$, 
it follows that
\[
\lim_{n\to\infty}h_{[n]}(x)=\left\{
\begin{array}{cc}
0,\ & {\rm if } \ h(x)\not=1,\\
1, \ & {\rm if } \ h(x)=1.\\
\end{array}\right.
\]
This is a contradiction with the fact that by Theorem \ref{T.MeanERg} the sequence $\{h_{[n]}\}_{n\in\N}$ converges to some function $g$ in $\cO_M(\R)$, where $g$ is obviuosly a continuous function on $\R$.

Since $h^{-1}(1)$ is an empty subset of $\R$, the  sequence $\{h_{[n]}\}_{n\in\N}$ necessarily converges to 0 in $\cO_M(\R)$ and hence, $P=0$. Accordingly,
$
\cS(\R)=\overline{{\rm Im}\, (I-M_h)}=\overline{{\rm Im}\, M_{1-h}}$.
	\end{proof}

In view of Proposition \ref{P.Funzioneh} we can now collect some examples.

\begin{ex} For a fixed $n\in\N$, we observe that the  function  $f_n(x)=\frac{1}{(1+x)^n}$, $x\not=-1$ belongs to  $C^\infty(\R\setminus\{-1\})$ and 
	\begin{equation}\label{eq.nfun}
	f_n^{(j)}(x)=\frac{(-1)^jn(n+1)\ldots (n+j)}{(1+x)^{n+j}}, \quad j\in\N, \ x\not=-1,
	\end{equation}
	as it is easy to verify. 
	
	(a) Consider the function $h(x):=\frac{1}{1+ae^{ix}}$, for $x\in\R$, with $|a|<1$. If we set $g(x):=ae^{ix}$ for $x\in\R$, then $h^n=f_n\circ g$ for every $n\in\N$. Therefore, by the Fa\`a Bruno formula and \eqref{eq.nfun}, we have for every $x\in\R$ and $j,n\in\N$ that 
	\[
	(h^n(x))^{(j)}=\sum\frac{j!}{k_1!\ldots k_j!}\frac{(-1)^kn(n+1)\ldots (n+k)}{(1+ae^{ix})^{n+k}}\frac{a^ji^je^{ijx}}{(1!)^{k_1}(2!)^{k_2}\ldots (j!)^{k_j}}, 
	\]
	where $k=k_1+k_2+\ldots+ k_j$ and $k_1+2k_2+\ldots+j k_j=j$. Accordingly, we have for every $x\in\R$ and $j,n\in\N$ that 
	\begin{align*}
	|(h^n(x))^{(j)}|&\leq \sum\frac{j!}{k_1!\ldots k_j!}\frac{n(n+1)\ldots (n+j)}{(1-|a|)^{n}}\frac{|a|^j}{(1!)^{k_1}(2!)^{k_2}\ldots (j!)^{k_j}}\\
	&=C_j\frac{n(n+1)\ldots (n+j)|a|^j}{(1-|a|)^n},
	\end{align*}
	being $C_j:=\sum\frac{j!}{k_1!\ldots k_j!}\frac{1}{(1!)^{k_1}(2!)^{k_2}\ldots (j!)^{k_j}}$. It follows for every $v\in\cS(\R)$ and  $l,n\in\N$ that 
	\[
	p_{l,v}(h^n)\leq C||v||_0\frac{n(n+1)\ldots (n+l)}{(1-|a|)^n}\leq CD||v||_0,
	\]
	where $C:=\sup_{j=0,\ldots, l} C_j<\infty$ and $D:=\sup_{n\in\N}\frac{n(n+1)\ldots (n+l)}{(1-|a|)^n}<\infty$. This means that $\{h^n\}_{n\in\N}$ is a bounded sequence of $\cO_M(\R)$. 
	
	(b) Consider the function $h(x):=\frac{a}{1+x^2}$ for $x\in\R$, with $|a|<1$. If we set $g(x):=x^2$ for $x\in\R$, then $h^n=a^nf_n\circ g$ for every $n\in\N$. Therefore, by the Fa\`a Bruno formula and \eqref{eq.nfun} taking into account that $g^{(j)}(x)=0$ for $j>2$, we get for every $x\in\R$ and $j,n\in\N$ that 
	\[
	(h^n(x))^{(j)}=a^n\sum\frac{j!}{k_1!k_2!}\frac{(-1)^kn(n+1)\ldots (n+k)}{(1+x^2)^{n+k}}(2x)^{k_1}. 
	\]
	Accordingly, we have for every $x\in\R$ and $j,n\in\N$ that 
	\[
	|(h^n(x))^{(j)}|\leq |a|^n\frac{n(n+1)\ldots (n+j)}{(1+x^2)^n}\sum\frac{j!}{k_1!k_2!}|2x|^{k_1}
	\]
	If we set $q_j(x):=\sum\frac{j!}{k_1!k_2!}|2x|^{k_1}$ for $x\in\R$, it follows for every  $l,n\in\N$ that 
	\[
	p_{l,v}(h^n)\leq C|a|^n n(n+1)\ldots (n+l)\leq CD,
	\]
	where $C:=\sup_{x\in\R}\sup_{j=0,\ldots,l}\frac{q_j(x)}{(1+x^2)^n}|v(x)|<\infty$ and $D:=\sup_{n\in\N}|a|^n n(n+1)\ldots (n+l)<\infty$. This means that the sequence $\{h^n\}_{n\in\N}$ is bounded in $\cO_M(\R)$.
	
	(c) In a similar way one shows that the sequence of the $n$-th powers of  the following functions is  bounded in $\cO_M(\R)$: $h(x):=ae^{-x^2}$, $k(x):=ae^{ix}$ and $s(x):=\frac{x}{1+x^2}$, for  $x\in\R$, with $|a|<1$. 
\end{ex}

A similar characterization of the power boundedness and mean ergodicity is valid for multiplication operators acting on  $C^\infty(\R)$. Indeed, we have

\begin{prop}\label{P.CInftyM} Let $h\in C^\infty(\R)$. Then the following properties are satisfied.
	\begin{enumerate}
		\item $M_h\colon C^\infty(\R)\to C^\infty(\R)$  is power bounded if and only if $\{h^n\}_{n\in\N}$ is a bounded sequence in $C^\infty(\R)$.
		\item $M_h\colon C^\infty(\R)\to C^\infty(\R)$  is mean ergodic if and only if $\{h_{[n]}\}_{n\in\N}$ is a convergent sequence in $C^\infty(\R)$.
	\end{enumerate}
	\end{prop}

\begin{proof} (1) If $M_h$ is power bounded, then the sequence $\{M_h^n\textbf{1}\}_{n\in\N}=\{h^n\}_{n\in\N}$ is necessarily bounded in $C^\infty(\R)$.
	
Conversely,  fixed a compact subset $K$ of $\R$ and an integer $m\in\N$, there exists $C_{K,m}>0$ such that
\[
\sup_{n\in\N}\sup_{x\in K}\sup_{0\leq i\leq m}|(h^n(x))^{(i)}|\leq C_{K,m}.
\]
This implies for every $n\in\N$ that 
\begin{align*}
\sup_{x\in K}\sup_{0\leq i\leq m}|(M_h^nf)^{(i)}(x)|& \leq \sup_{x\in K}\sup_{0\leq i\leq m}\sum_{j=0}^i\binom{i}{j}|(h^n(x))^{i-j}||f^{(j)}(x)|\\& \leq 2^mC_{K,m}\sup_{x\in K}\sup_{0\leq i\leq m}|f^{(i)}(x)|.
\end{align*}
Since $K$ and $m$ are arbitrary, this means that the operator $M_h$ is power bounded.

(2) follows in a similar way.
\end{proof}

We now pass to investigate the ergodic properties of the convolution operators acting on $\cS(\R)$. We see that the characterization of such properties follow as 
  consequences of Theorems \ref{T.PowerB} and \ref{T.MeanERg}. But first, we need to observe the following facts.

Iterating the procedures in \eqref{eq.FTC} and in \eqref{eq.FTD},  it follows  for every $n\in\N$,  $f\in\cS(\R)$ and $S\in \cO'_C(\R)$ that
\begin{equation}\label{eq.FTCN}
\cF(C^n_T(f))=\cF(T)^n\hat{f}, \quad \cF(\cC^n_T(S))=\cF(T)^n\cF(S).
\end{equation}
This means  that 
\begin{equation}\label{eq.Compo}
\cF\circ C^n_T=M^n_{\cF(T)}\circ \cF\quad (\cF\circ \cC^n_T=M^n_{\cF(T)}\circ \cF, {\rm \ resp.}), \quad n\in\N.
\end{equation}
If we  set
\[
(\star T)^2:=T\star T,\quad (\star T)^n:=T\star (\star T)^{n-1} \ ({\rm for }\ n\geq 2)
\]
(the definition is well-posed because $T, (\star T)^{n-1}\in \cO'_C(\R)$ and hence $(\star T)^n=T\star (\star T)^{n-1}\in  \cO'_C(\R)$), then by  \eqref{eq.FTD} we get for every $n\in\N$ that  
\begin{equation}\label{eq.StarN}
\cF((\star T)^n)=\cF(T)^n,
\end{equation}
and hence, by \eqref{eq.FTCN} that 
\[
\cF(\cC_T^n(\delta))=\cF(T)^n\cF(\delta)=\cF((\star T)^n)\ \Rightarrow \cC_T^n(\delta)=(\star T)^n,
\]
being $\delta$ the delta of Dirac. 
Since $(\star T)^n\in \cO_C'(\R)$ for every $n\in\N$, we also get for every  $n\in\N$ that  the distribution $(\star T)_{[n]}:=\frac{1}{n}\sum_{m=1}^{n}(\star T)^m\in \cO_C'(\R)$  and that 
\[
(\cC_T)_{[n]}(\delta)=(\star T)_{[n]}.
\]

\begin{prop}\label{P.PBC} Let $T\in \cO_C'(\R)$. Then the following properties are equivalent.
	\begin{enumerate}
		\item[(1)] $C_T$ is power bounded on $\cS(\R)$.
		\item[(1)'] $\cC_T$ is power bounded on $\cS'(\R)$.
		\item[(2)] $C_T$ is power bounded on $\cO_C(\R)$.
		\item[(2)'] $\cC_T$ is power bounded on $\cO'_C(\R)$.
		\item[(3)] The sequence  $\{(\star T)^n\}_{n\in\N}$ is  bounded in $\cO'_C(\R)$.
	\end{enumerate}	
\end{prop}

\begin{proof} (1)$\Leftrightarrow$(1)' follows from \cite[Lemma 2]{ABR-2}, after having observed that $\cC_T=C_T'$ and  $\cS(\R)$ is a reflexive lcHs.
	
(2)$\Leftrightarrow$(2)' follows by the same arguments above.

(1)$\Rightarrow$(2)'. Since the Fourier transform $\cF$ is a topological isomorphism from $\cS(\R)$ onto itself,  thanks to  \eqref{eq.Compo} the assumption implies that the  operator $M_{\cF(T)}\colon \cS(\R)\to \cS(\R)$  is also power bounded. By Theorem \ref{T.PowerB} this is equivalent to the fact that  the operator  $M_{\cF(T)}\colon \cO_M(\R)\to \cO_M(\R)$ is power bounded. So, it follows via \eqref{eq.Compo} that the convolution operator $\cC_T\colon \cO_C'(\R)\to \cO'_C(\R)$ is power bounded too, taking in account that the Fourier transfom $\cF$ is also a topological isomorphism from $\cO_C'(\R)$ onto $\cO_M(\R)$.  

(2)'$\Rightarrow$(3). The delta  $\delta$ of Dirac  belongs to $\cO'_C(\R)$ and so, by assumption, $\{\cC_T^n(\delta)\}_{n\in\N}=\{(\star T)^n\}_{n\in\N}$ is a bounded sequence of $\cO_C'(\R)$.

(3)$\Rightarrow$(1). The facts that the Fourier transform $\cF$ is a topological isomorphism from $\cO'_C(\R)$ onto $\cO_M(\R)$ and that $\{(\star T)^n\}_{n\in\N}$ is a bounded sequence of $\cO_C'(\R)$ imply via \eqref{eq.StarN} that $\{\cF(T)^n\}_{n\in\N}$ is a bounded sequence of $\cO_M(\R)$.  By Theorem \ref{T.PowerB} it follows that the operator $M_{\cF(T)}\colon \cS(\R)\to \cS(\R)$ is power bounded.  Hence, this implies via \eqref{eq.Compo} that the  operator $C_T\colon \cS(\R)\to \cS(\R)$ is necessarily  power bounded (taking in account that the Fourier transform $\cF$ is a topological isomorphism from $\cS(\R)$ onto itself).	\end{proof}

\begin{prop}\label{P.PBM} Let $T\in \cO_C'(\R)$. Then the following properties are equivalent.
	\begin{enumerate}
		\item[(1)] $C_T$ is (uniformly) mean ergodic on $\cS(\R)$.
		\item[(1)'] $\cC_T$ is (uniformly) mean ergodic on $\cS'(\R)$.
		\item[(2)] $C_T$ is (uniformly) mean ergodic on $\cO_C(\R)$.
		\item[(2)'] $\cC_T$ is (uniformly) mean ergodic on $\cO'_C(\R)$.
		\item[(3)] The sequence  $\{(\star T)_{[n]}\}_{n\in\N}$ is  convergent in  $\cO'_C(\R)$.
	\end{enumerate}	
\end{prop}

\begin{proof} (1)$\Leftrightarrow$(1)' follows from \cite[Lemma 2.1]{ABR-1}, after having observed that  $\cC_T=C_T'$ and $\cS(\R)$ is a reflexive lcHs space (and hence, $\cS'(\R)$ is barrelled reflexive lcHs).
	
(2)$\Leftrightarrow$(2)' follows by the same arguments above.

(1)$\Rightarrow$(2)'. We first observe that from \eqref{eq.Compo} it follows for every $n\in\N$ that
\begin{equation}\label{eq.MER}
\cF\circ (C_T)_{[n]}=(M_{\cF(T)})_{[n]}\circ \cF, \quad \cF\circ (\cC_T)_{[n]}=(M_{\cF(T)})_{[n]}\circ \cF.
\end{equation}	
Since the Fourier transform $\cF$ is a topological isomorphism from $\cS(\R)$ onto itself, the assumption together with  \eqref{eq.MER} implies  that the operator $M_{\cF(T)}\colon \cS(\R)\to \cS(\R)$ is (uniformly) mean ergodic. By Theorem \ref{T.MeanERg} this is equivalent to the fact that the operator $M_{\cF(T)}\colon \cO_M(\R)\to \cO_M(\R)$ is (uniformly) mean ergodic. So, again via \eqref{eq.MER}, it follows that the convolution operator $\cC_T\colon \cO_C'(\R)\to \cO_C'(\R)$ is (uniformly) mean ergodic, after having observed that the Fourier transform $\cF$ is also a topological isomorphism from $\cO'_C(\R)$ onto $\cO_M(\R)$.
	
(2)'$\Rightarrow$(3). The delta $\delta$ of Dirac belongs to $\cO'_C(\R)$ and so, by assumption, $\{(\cC_T)_{[n]}\delta\}=\{(\star T)_{[n]}\}$ is a convergent sequence in $\cO'_C(\R)$.

(3)$\Rightarrow$(1). The facts that the Fourier transform $\cF$ is a topological isomorphism from $\cO'_C(\R)$ onto $\cO_M(\R)$ and that $\{(\star T)_{[n]}\}_{n\in\N}$ is a convergent sequence in $\cO'_C(\R)$ imply thanks to  \eqref{eq.MER} that $\{\cF(T)_{[n]}\}_{n\in\N}$ is a convergent sequence in $\cO_M(\R)$.  By Theorem \ref{T.MeanERg} it follows that the operator $M_{\cF(T)}\colon \cS(\R)\to \cS(\R)$ is (uniformly) mean ergodic. Hence, this together with  \eqref{eq.MER} implies  that the operator $C_T\colon \cS(\R)\to \cS(\R)$ is necessarily (uniformly) mean ergodic (taking into account that the Fourier transform $\cF$ is a topological isomorphism from $\cS(\R)$ onto itself).
\end{proof}

\begin{rem} Let $T\in \cO_C'(\R)$. If $C_T\colon \cS(\R)\to \cS(\R)$ is either power bounded or mean ergodic, then $||\cF(T)||_0\leq 1$ and $(\cF(T))^{-1}(1)=\emptyset$. Moreover, $(\star T)_{[n]}\to 0$ in $\cO_C'(\R)$ as $n\to\infty$. This  easily follows from Proposition \ref{P.Funzioneh} thanks to the identity $\cF\circ C_T=M_{\cF(T)}\circ \cF$, being $\cF$ a topological isomorphism from $\cO'_C(\R)$ onto $\cO_M(\R)$.
\end{rem}
	
	Finally, we point out that a similar characterization of the power boundedness and mean ergodicity is valid also for convolution operators acting on the strong dual $\cE'(\R)$ of $C^\infty(\R)$. To state and prove the result,  we  observe that the same  arguments at the end of \S 2 show that if $T\in \cE'(\R)$, then  $(\star T)^n$ is well-posed and belongs to $\cE'(\R)$ for any $n\in\N$. Hence, $(\star T)^{[n]}$ is also well posed  and belongs to $\cE'(\R)$ for any $n\in\N$.
	
	\begin{prop}\label{P.CInftyC} Let $T\in \cE'(\R)$. Then the following properties are satisfied.
		\begin{enumerate} 
			\item $\cC_T\colon \cE'(\R)\to \cE'(\R)$ is power bounded if and only if $\{(\star T)^n\}_{n\in\N}$ is a bounded sequence of $\cE'(\R)$.
			\item $\cC_T\colon \cE'(\R)\to \cE'(\R)$ is mean ergodic if and only if $\{(\star T)_{[n]}\}_{n\in\N}$ is a convergent  sequence of $\cE'(\R)$.
		\end{enumerate}
	\end{prop}
	
\begin{proof} (1) If $\cC_T$ is power bounded, then the sequence $\{\cC^n_T(\delta)\}_{n\in\N}=\{(\star T)^n\}_{n\in\N}$ is necessarily bounded in $\cE'(\R)$.
	
	Conversely, there exist a compact subset $K$ of $\R$,  $m\in\N$ and $C_{K,m}>0$ such that
	\[
	\sup_{n\in\N}\sup_{f\in U_{K,m}}|\langle (\star T)^n, f\rangle|\leq C_{K,m},
	\]
	where $U_{K,m}:=\{f\in C^\infty(\R)\colon ||f||_{K,m}:=\sup_{x\in K}\sup_{0\leq i\leq m}|f^{(i)}(x)|\}$. So,  it follows for every $i,n\in\N$ and $f\in C^\infty(\R)$ that
	\[
	| ((\star T)^n\star f)^{(i)}(x)|=|\langle (\star T)^n, (\check{\tau_x f})^{(i)}\rangle|\leq C_{K,m}||(\check{\tau_x f})^{(i)}||_{K,m}, \quad x\in\R.
	\]
	Therefore, for a fixed compact subset $H$ of $\R$ and $l\in\N$, this yields for every $n\in\N$ and $f\in C^\infty(\R)$ that 
	\[
	\sup_{x\in H}\sup_{0\leq i\leq l}| ((\star T)^n\star f)^{(i)}(x)|\leq C_{K,m}\sup_{z\in H-K}\sup_{0\leq i\leq l+m}|f^{(i)}(z)|.
	\]
Since $H$ and $l\in\N$ are arbitrary, this implies that the set $B^*:=\{(\star T)^n\star f\colon f\in B \}$ is bounded in $C^\infty(\R)$ whenever $B$ is a bounded subset of  $C^\infty(\R)$. Accordingly, for   a fixed  bounded subset $B$ of $C^\infty(\R)$ there exists $B^*$ bounded subset of $C^\infty(\R)$ such that every $n\in\N$ and $S\in \cE'(\R)$ we have
\[
\sup_{f\in B}|\langle \cC_T^n(S),f\rangle|=\sup_{f\in B}\langle S, \check{(\star T)^n}\star f\rangle|\leq \sup_{g\in B^*}|\langle S, g\rangle |, 
\]
i.e., $\cC_T$ is power bounded.

(2) If $\cC_T$ is mean ergodic, then the sequence $\{(\cC_T)_{[n]}(\delta)\}_{n\in\N}=\{(\star T)_{[n]}\}_{n\in\N}$ is convergent in $\cE'(\R)$.

	Conversely, there exist $T_0\in \cE'(\R)$ (say, $T_0=0$), a compact subset $K$ of $\R$ and  $m\in\N$ such that
\[
||(\star T)_{[n]}||'_{K,m}:=\sup_{f\in U_{K,m}}|\langle (\star T)^n, f\rangle|\to 0.
\]
So, arguing as in part (1), it follows for every $i,n\in\N$ and $f\in C^\infty(\R)$ that
\[
| ((\star T)_{[n]}\star f)^{(i)}(x)|\leq ||(\star T)_{[n]}||'_{K,m}||(\check{\tau_x f})^{(i)}||_{K,m}, \quad x\in\R.
\]
Therefore, for a fixed compact subset $H$ of $\R$ and $l\in\N$, this yields for every $n\in\N$ and $f\in C^\infty(\R)$ that 
\begin{equation}\label{eq.CONV}
\sup_{x\in H}\sup_{0\leq i\leq l}| ((\star T)_{[n]}\star f)^{(i)}(x)|\leq ||(\star T)_{[n]}||'_{K,m}\sup_{z\in H-K}\sup_{0\leq i\leq l+m}|f^{(i)}(z)|.
\end{equation}
Since $H$ and $l\in\N$ are arbitrary, this impies that $(\star T)_{[n]}\star f\to 0$ in $C^\infty(\R)$ as $n\to\infty$ whenever $f\in C^\infty(\R)$. 

Now, let $S\in \cE'(\R)$. Then there exist a compact subset $H$ of $\R$, $l\in\N$ and $C>0$ such that 
\begin{equation}\label{eq.StimaC}
|\langle S,f\rangle |\leq C \sup_{x\in H}\sup_{0\leq i\leq l}|f^{(i)}(x)|
\end{equation}
for every $f\in C^\infty(\R)$. So, combining   \eqref{eq.CONV} with \eqref{eq.StimaC} we obtain for every $n\in\N$ and $f\in C^\infty(\R)$ that 
\[
|\langle (\cC_T)_{[n]}(S),f\rangle|=|\langle S, \check{(\star T)_{[n]}}\star f\rangle|\leq C ||(\star T)_{[n]}||'_{K,m}\sup_{z\in H-K}\sup_{0\leq i\leq l+m}|f^{(i)}(z)|.
\]
So, for a fixed bounded subset $B$ of $C^\infty(\R)$, it follows for every $n\in\N$ that 
\[
\sup_{f\in B}|\langle (\cC_T)_{[n]}(S),f\rangle|\leq C ||(\star T)_{[n]}||'_{K,m}\sup_{f\in B}\sup_{z\in H-K}\sup_{0\leq i\leq l+m}|f^{(i)}(z)|.
\]
Since $\sup_{f\in B}\sup_{z\in H-K}\sup_{0\leq i\leq l+m}|f^{(i)}(z)|<\infty$ and $||(\star T)_{[n]}||'_{K,m}\to 0$ as $n\to\infty$, this implies that $\sup_{f\in B}|\langle (\cC_T)_{[n]}(S),f\rangle|\to 0$ as $n\to\infty$. But $B$ is arbitrary. Then we can conclude that $(\cC_T)_{[n]}\to 0$ in $\cE'(\R)$ as $n\to\infty$, i.e., $\cC_T$ is mean ergodic.

\begin{rem} Since $\cO_M(\R)$ is continuously included in $C^\infty(\R)$ and $\cE'(\R)$ is continuously included in $\cO'_C(\R)$,  the  results of this section clearly imply that:
	\begin{enumerate}
		\item Let  $h\in \cO_M(\R)$. If $M_h\colon \cS(\R)\to \cS(\R)$ is power bounded (mean ergodic, resp.), then  $M_h\colon C^\infty(\R)\to C^\infty(\R)$ is power bounded (mean ergodic, resp.), while
		\item Let $T\in \cE'(\R)$. If $\cC_T\colon \cE'(\R)\to \cE'(\R)$ is  power bounded (mean ergodic, resp.), then   $\cC_T\colon \cS'(\R)\to \cS'(\R)$ is  power bounded (mean ergodic, resp.).
			\end{enumerate}
	We point out that  the proof of part (1) is also given in the course of  the proof of Theorems \ref{T.PowerB} and \ref{T.MeanERg}.
\end{rem}
\end{proof}

\section{Appendix}

We  establish here some general results on the spectrum of continuous linear operators acting on  Fr\'echet spaces.

\begin{prop}\label{P.limite} Let $E$ be a Fr\'echet space and $T\in \cL(E)$. If there exists $\lim_{n\to\infty}\frac{T^n}{n}=0$ in $\cL_s(E)$, then $\sigma(T)\subseteq \overline{\mathbb D}$.
\end{prop} 

\begin{proof} The assumption  $\tau_s$-$\lim_{n\to\infty}\frac{T^n}{n}=0$ implies that the set $\{\frac{T^n}{n}\colon n\in\N\}$ is equicontinuous in $E$. If $\{p_j\}_{j\in\N}$ is an increasing sequence of continuous seminorms generating the lc-topology of $E$, then for each $j\in\N$ there exist $j'\in\N$ with $j'\geq j$ and $c_j>0$ such that
	\begin{equation}\label{eq.equic}
	p_j\left(\frac{T^nx}{n}\right)\leq c_jp_{j'}(x)
	\end{equation}
	for every $x\in E$ and $n\in\N$. 
	
	For every  $j\in\N$ and $x\in E$ we define 
	\begin{equation}
	q_j(x):=\max\left\{p_j(x), \sup_{n\in\N}p_j\left(\frac{T^n x}{n}\right)\right\}. 
	\end{equation}
	Thus, each $q_j$ is a seminorm on $E$ and $q_j\leq q_{j+1}$ for every $j\in\N$, as it is easy to show. On the other hand,  from  \eqref{eq.equic} it follows that 
	\begin{equation}\label{eq.ineq}
	p_j(x)\leq q_j(x)\leq c_jp_{j'}(x)
	\end{equation}
	for every $j\in\N$ and $x\in E$. Therefore, $\{q_j\}_{j\in\N}$ is an increasing sequence of continuous seminorms generating the lc-topology of $E$. Moreover,  we have for every  $j, n\in\N$ and $x\in E$ that
	\begin{equation}\label{eq.equic1}
	q_j\left(\frac{T^nx}{n}\right)=\max\left\{p_j\left(\frac{T^nx}{n}\right),\sup_{m\in\N}\left(\frac{T^{n+m}x}{nm}\right)\right\}\leq 2q_j(x),
	\end{equation}
	as $2nm\geq n+m$ for all $n,m\in\N$. This yields for every  $j, n\in\N$ and $x\in E$ that
	\[
	q_j(T^nx)\leq 2n q_j(x)
	\]
	and hence,
	\[
	\sup_{q_j(x)\leq 1}q_j(T^nx)\leq 2n. 
	\]
	Letting $n\to\infty$, it follows for every $j\in\N$ that 
	\begin{equation}\label{eq.radius}
	\limsup_{n\to\infty}\sqrt[n]{\sup_{q_j(x)\leq 1}q_j(T^nx)}\leq \lim_{n\to\infty}\sqrt[n]{2n}=1.
	\end{equation}
	
	Now, we fix $\lambda\in \C$ with $|\lambda|>1$ and choose $0<c<1$ so that $c|\lambda|>1$.  Consequently, for a fixed $j\in\N$, we have by \eqref{eq.radius} that 
	\[
	\limsup_{n\to\infty}\sqrt[n]{\sup_{q_j(x)\leq 1}q_j(T^nx)}\leq 1 <c|\lambda|.
	\]
	So, there exists $n_0\in\N$ such that for every $x\in X$ and $n\geq n_0$ we have
	\[
	q_j(T^nx)<c^n|\lambda|^nq_j(x).
	\]
	This implies for every $x\in\N$ and $n\geq n_0$  that
	\[
	q_j\left(\sum_{m=n}^\infty\frac{T^mx}{\lambda^m}\right)\leq \left(\sum_{m=n}^\infty c^m\right)q_j(x),
	\]
	where $\sum_{m=n}^\infty c^m\to 0$ for $n\to\infty$ as $0<c<1$. Accordingly, as $j\in\N$ is arbitrary, the series $\sum_{n=0}^\infty\frac{T^n}{\lambda^n}$ is convergent in $\cL_s(E)$ and so, the operator $R_\lambda:=\frac{1}{\lambda}\sum_{n=0}^\infty\frac{T^n}{\lambda^n}\in\ \cL(E)$. In particular, $R_\lambda (\lambda I-T)=(\lambda I -T)R_\lambda=I$. Thus,  $\lambda\in \rho(T)$ and $R(\lambda,T)=R_\lambda$.
\end{proof}

\begin{cor}\label{C.PowerBounded} Let $E$ be a Fr\'echet space and $T\in \cL(E)$. If $T$ is power bounded, then $\sigma(T)\subseteq \overline{\mathbb D}$.
\end{cor}

\begin{proof} The assumption on $T$ clearly implies that there exists $\lim_{n\to\infty}\frac{T^n}{n}=0$ in $\cL_s(E)$. So, the result follows from Proposition \ref{P.limite}.
\end{proof}

\begin{rem} Let $E$ be a sequentially complete barrelled lcHs and $T\in \cL(E)$. If there exists $\lim_{n\to\infty}\frac{T^n}{n}=$ in $\cL_s(E)$,  arguing as in the proof of Proposition \ref{P.limite} we can show that $\sigma(T)\subseteq \overline{\D}$. 
\end{rem}

The next aim is to extend in the setting of separable Fr\'echet spaces a result of Jamison \cite{Jamison} about the size of $\sigma_p(T)\cap \T$ with $T$ a  power bounded operator. In order to do this, we observe that if $T$ is a power bounded operator acting on a Fr\'echet space $E$, then there exists an increasing sequence $\{q_j\}_{j\in\N}$ of continuous seminorms generating the lc-topology of $E$ such that for every $j,n\in\N$ and $x\in E$ we have
\begin{equation}\label{eq.EquiPB}
q_j(T^n x)\leq q_j(x).
\end{equation}
Indeed, if $\{p_j\}_{j\in\N}$ is an increasing sequence of continuous seminorms generating the lc-topology of $E$, then for each $j\in\N$ there exist $j'\in\N$ with $j'\geq j$ and $c_j>0$ such that
\[
p_j(T^nx)\leq c_jp_{j'}(x)
\]
for every $x\in E$ and $n\in\N$. If we set
\[
q_j(x):=\sup_{n\in\N_0}p_j(T^nx)
\]
for every $j\in\N$ and $x\in E$, then $\{q_j\}_{j\in\N}$ is an increasing sequence of  continuous seminorms generating the lc-topology of $E$ for which  \eqref{eq.EquiPB} is satisfied. The proof is achieved by arguing as in the proof of Proposition \ref{P.limite}.

\begin{lem}\label{L.Modul} Let $E$ be a Fr\'echet space and $T\in\cL(E)$ be a power bounded operator satisfying \eqref{eq.EquiPB}.  Let $\lambda_1,\lambda_2\in \sigma_p(T)\cap\T$ be independent. Then, if $x_1$ and $x_2$ are eigenvectors for $\lambda_1$ and $\lambda_2$ respectively, with $q_j(x_1)=q_j(x_2)=1$ for some $j\in\N$, we have $q_j(x_1-x_2)\geq 1$.
\end{lem}

\begin{proof} We consider the quotient space $X/\ker p_j$ and endow it with the canonical quotient norm $\hat{p}_j$ defined by
	\[
	\hat{p}_j(Q_jx):=\inf\{p_j(y)\colon x-y\in\Ker q_j\},
	\]
	where $Q_j\colon E\to \frac{E}{\Ker q_j}$ denotes the canonical quotient map. Then $\left(\frac{E}{\Ker q_j},\hat{p}_j \right)$ is a normed space. The operator $T$ induces a continuous linear operator $\hat{T}_j$ acting on $\frac{E}{\Ker q_j}$ via \eqref{eq.EquiPB} such that 
	\[
	Q_j\circ T=T_j\circ Q_j \ \ {\rm and } \ \ \hat{p}_j(T_j\hat{x})\leq \hat{p}_j(\hat{x}) \ \ {\rm for\ all} \ \hat{x}\in \frac{E}{\Ker q_j}.
	\]
	Therefore, $T_j$ can be continuously extended on the completion $E_j$ of $\left(\frac{E}{\Ker q_j},\hat{p}_j \right)$, denoted again by $T_j$, for which $Q_j\circ T=T_j\circ Q_j$ continues to hold.
	
	We now observe that
	\[
	T_j(Q_jx_i)=Q_j(Tx_i)=Q_j(\lambda_ix_i)=\lambda_i Q_j(x_i), \quad i=1,2, 
	\]
	and that, setting $\hat{x}_i=Q_j(x_i)$ for $i=1,2$, we have
	\[
	\hat{q_j}(\hat{x}_i)=1,\quad i=1,2.
	\]
	By \cite[Lemma 2]{Jamison} we may conclude that $\hat{q_j}(\hat{x}_1-\hat{x}_2)\geq 1$ and hence, 
	\[
	1\leq \hat{q_j}(\hat{x}_1-\hat{x}_2)\leq q_j(x_1-x_2).
	\]
\end{proof}

\begin{prop}\label{P.Modulo} Let $E$ be a separable  Fr\'echet space and $T\in\cL(E)$ be a power bounded operator satisfying \eqref{eq.EquiPB}. Then $\sigma_p(T)\cap \T$ is at most countable.
\end{prop}

\begin{proof} We set $\Gamma:=\sigma_p(T)\cap \T$ and suppose  that $\Gamma$ is uncountable. By Lemma \ref{L.Modul} $\Gamma$ contains an uncountable set $\Lambda$ such that any two distint elements of $\Lambda$ are independent. For each $\lambda\in \Lambda$ we select an eigenvector $x_\lambda$ corresponding to the eigenvalue $\lambda$ and we set $D:=\{x_\lambda\colon \lambda\in \Lambda\}$.  Of course, $D$ is uncountable. If for each $j\in\N$ we define
	\[
	D_j:=\{x\in D\colon q_j(x)\not =0\},
	\]
	then $D=\cup_{j\in\N}D_j$. Moreover, we may suppose without loss of generality for every $j\in\N$  that each element of $D_j$ satisfies $q_j(x)=1$.
	
	Since $D$ is uncountable, there exists $j_0\in\N$ such that $D_{j_0}$ is uncountable. We now set $U_{j_0}:=\{x\in E\colon q_{j_0}(x)<1\}$ and $U_{x,j_0}(r):=x+rU_{j_0}$ for $x\in E$ and $r>0$. Then any two distint elements of $\{U_{x,j_0}(\frac{1}{2})\colon x\in D_{j_0}\}$ are disjoint sets by Lemma \ref{L.Modul}. This clearly contradicts the separability of $E$. Thus $\Gamma$ is at most countable. 
\end{proof}

\begin{cor}\label{C.Modulo} 
	Let $E$ be a separable  Fr\'echet space and $T\in\cL(E)$. If $\sigma_p(T)\cap \T$ is uncountable, then $T$ cannot be power bounded.
\end{cor}

\begin{proof} We suppose that  $T$ is a power bounded operator. Then there exists an increasing sequence $\{q_j\}_{j\in\N}$ of continuous seminors generating the lc-topology of $E$ such that $T$ satisfies  \eqref{eq.EquiPB} (see the comments before Lemma \ref{L.Modul}). So, we can apply Proposition \ref{P.Modulo} to conclude that  $\sigma_p(T)\cap \T$ is at most countable. This is a contradiction and so the proof is complete. \end{proof}


\end{document}